\documentclass[a4paper,12pt]{amsart}

\usepackage{xcolor}

\usepackage{amssymb,amscd,amsthm,mathrsfs}
\usepackage[backref=page]{hyperref} 
\usepackage[ansinew]{inputenc}
\usepackage[centertags]{amsmath}
\usepackage{graphicx,psfrag}

\renewcommand*{\backref}[1]{}
\renewcommand*{\backrefalt}[4]{\quad \tiny
  \ifcase #1 (\textbf{NOT CITED.})%
  \or    (Cited on page~#2.)%
  \else   (Cited on pages~#2.)%
  \fi}

\newtheorem*{thm*}{Theorem}
\newtheorem*{thmA}{Theorem A}
\newtheorem*{thmB}{Theorem B}
\newtheorem*{thmC}{Theorem C}
\newtheorem*{thmD}{Theorem D}

\newtheorem{thm}{Theorem}[section]

\newtheorem{lemma}[thm]{Lemma}
\newtheorem{lema}[thm]{Lemma}
\newtheorem{prop}[thm]{Proposition}

\newcommand{\bi}{\begin{itemize}}
\newcommand{\ei}{\end{itemize}}

\theoremstyle{definition}

\theoremstyle{remark}
\newtheorem{obs}[thm]{Remark}

\newcommand{\lqqd}{\par\hfill {$\Box$} \vspace*{.05in}}

\newcommand{\D}{\mathbb{D}}
\newcommand{\HH}{\mathbb{H}}
\newcommand{\T}{\mathbb{T}}
\newcommand{\R}{\mathbb{R}}
\newcommand{\Z}{\mathbb{Z}}
\newcommand{\N}{\mathbb{N}}

\newcommand{\cC}{\mathcal{C}}
\newcommand{\F}{\mathcal{F}}
\newcommand{\cA}{\mathcal{A}}
\newcommand{\en}{\subset}
\newcommand{\cU}{\mathcal{U}}

\newcommand{\C}{\mathcal{C}}

\newcommand{\A}{\mathbb{A}}

\newcommand{\eps}{\varepsilon}

\newcommand{\az}[1]{#1}

\newcommand{\ro}[1]{#1}

\newcommand{\ve}[1]{#1}

\author[A. Passeggi]{Alejandro Passeggi}
\address{UdelaR, Facultad de Ciencias.}
\curraddr{Igua 4225 esq. Mataojo. Montevideo, Uruguay.}
\email{alepasseggi@gmail.com}
\author[R. Potrie]{Rafael Potrie}
\address{UdelaR, Facultad de Ciencias.}
\curraddr{Igua 4225 esq. Mataojo. Montevideo, Uruguay.}
\email{rpotrie@cmat.edu.uy}
\author[M. Sambarino]{Mart\'{\i}n Sambarino}
\address{UdelaR, Facultad de Ciencias.}
\curraddr{Igua 4225 esq. Mataojo. Montevideo, Uruguay.}
\email{samba@cmat.edu.uy}

\title[\tiny{Rotation sets and entropy on attracting annular continua}]{Rotation intervals and entropy on attracting annular continua}
\thanks{The authors were partially supported by CSIC Grupo 618 "Sistemas Din\'{a}micos". The second and third authors were partially supported by Palis-Balzan's project.}

\begin{document}

\maketitle

\begin{abstract}

We show that if $f$ is an annular homeomorphism admitting an attractor which is an irreducible annular continua with two different rotation numbers, then the entropy of $f$ is positive. Further, the entropy is shown to be associated to a $C^0$-robust \emph{rotational horseshoe}. On the other hand, we construct examples of annular homeomorphisms with such attractors so that the rotation interval is uniformly large but the entropy approaches zero as much as desired.

\smallskip

The developed techniques allow us to obtain similar results in the context of \emph{Birkhoff attractors}.

\end{abstract}

\section{Introduction}\label{SectionIntroduccion}

The study of annular dynamics goes back at least to Poinc\ro{a}r\'{e} who \az{used suitable (Poincar\'{e}) sections in the restricted three body problem to reduce the initial dynamics to an annulus.} 
This study turned out to be crucial in understanding the problem of stability (see \cite{CheicinerSurvey}) and gave rise to what nowadays is known as KAM theory \cite{surveyKAM}.

\smallskip

In this theory the considered dynamics are \az{volume preserving}, reflecting the \ro{conservation} laws \ro{of} the particular mechanical system. On the other hand, when physical problems involving non-conservative forces are \az{analised}, sometimes one is lead to study dissipative versions of the former class of systems (see for instance \cite{pendulos,entrchaos,nonlinearandchaos}). \ro{In this setting \emph{strange attractors} emerge as natural objects related to the underlying dynamics (for the definitions and basic examples see \cite{scholarmilnor})}. These were proved to exist by Birkhoff \cite{Birkhofforig}, who actually showed that they appear associated to the wide class of differentiable annular maps given by \emph{dissipative twist maps} (see \cite{lecalvez} for a comprehensive exposition). \ro{They were also found numerically by} R. Shaw, associated to the dynamics induced by differential equations such as the \emph{forced Van der Pol systems}
\cite{shawatt,nonlinearandchaos}. Since then, annular attractors have been studied both from the mathematical and physical point of view (see \cite{lecalvez,nonlinearandchaos}). 

\smallskip

In order to study this kind of attractors, there are two important dynamical invariants: the \emph{rotation set} and the \emph{topological entropy}. The former is given by averages of displacements of points in the attractor, information \ro{that is expressed by an interval of real numbers} (see below). The \ro{latter is} a quantity which measures how \emph{chaotic} the attractor is \footnote{A weaker version is the study of existence of positive \emph{Lyapunov exponents} -when the dynamics is smooth- which is implied by positivity of topological entropy.}. It is then natural to try to understand whether these two invariants are related and this motivates our article: we prove that a non-trivial rotation set implies positive topological entropy, and, in contrast, provide examples of systems which have uniformly large rotation intervals and arbitrary small topological entropy. 

\smallskip

From the pure mathematical point of view, this problem can be thought of as a version of the well known Shub's entropy conjecture for maps in the homotopy class of the identity (see \cite{ShubICM}): some geometric property of the dynamical system detectable from ``large scale'' imposes some lower bound on its complexity (e.g. topological entropy). In this case we focus on the rotation set of a dynamical system (see \cite{franksICM}), motivated by previous results providing a relationship between the \emph{shape and size} of this set and the topological entropy in some particular settings (degree one circle maps, torus homeomorphisms isotopic to the identity). Searching for similar relationships in the setting of dissipative annular homeomorphisms, we came into a rather surprising outcome: it is possible to show positive entropy assuming that the rotation set is non-trivial, yet, it is not possible to obtain lower bounds depending on the shape and size of the rotation set.

\ro{The following subsection presents} an account of the results in this paper to prepare for the precise statements.

\subsection{Presentation of the results}

The \emph{rotation set} is an invariant for dynamical systems which has been shown to contain \ro{essential} information of the dynamics when the underlying space has low dimension,
in particular in dimensions one and two.

Poincare's theory for orientation preserving homeomorphisms on the circle is the paradigmatic case: \ro{the rotation number}
turns out to be a number which provides a complete description of the underlying dynamics (see for example \cite[Chapter 11]{KatokHasselblatt}).
Still in dimension one, there is a natural generalisation of the \ro{rotation number} for degree one endomorphisms of the circle,
given by an (\az{possibly} trivial) interval called \emph{rotation set}. From this set crucial information of the dynamics can be deduced,
\ro{providing, for instance, criteria for the} existence of periodic orbits with certain relative displacements among other interesting properties
(see \cite{circulo} and references therein).

In dimension two, the dynamics of certain surface homeomorphisms homotopic to the identity is usually described by means of this topological invariant. In particular, \ro{for} the annulus $\A=\mathbb{S}^1\times\R$ and the two-torus $\T^2=\R^2/_{\Z^2}$ it can be said \ro{that} a theory has been built supported on the rotation set. In these contexts, for a dynamics $f$ given by a homeomorphisms in the homotopy class of the identity \ro{and any compact, forward invariant set $K$}, the rotation set associated to a lift $F:\R^2\to\R^2$ is defined as 

\small

$$\rho_K(F)=\left\{\lim_k \frac{\pi_1(F^{n_k}(x_k)-x_k)}{n_k}\ |\ x_k\in \ro{\pi^{-1}(K)},\ n_k \nearrow +\infty \right\} \ro{\subset \R}\ ,\mbox{ and  }$$
$$\rho_K(F)=\left\{\lim_k \frac{F^{n_k}(x_k)-x_k}{n_k}\ |\ x_k\in \ro{\pi^{-1}(K)},\ n_k \nearrow +\infty \right\} \ro{\subset \R^2} ,$$

\normalsize

\noindent respectively, \ro{where $\pi$ denotes for both cases the quotient map} and $\pi_1$ is the projection over the first coordinate \ro{in $\R^2$}. \ro{In case $K=\T^2$ one writes $\rho(F)$ instead $\rho_K(F)$.}

\smallskip

When $K\subset\A$ is also connected, the shape of this set is given by an (possibly degenerated) interval in the annular case. For the toral case
the foundational result by Misiurewicz and Zieman \cite{MZ}, shows that $\rho(F)$ is a (possibly degenerated) compact and convex set.
From these facts, there exists a vast list of interesting results, where assuming possible geometries for the rotation set, descriptions of the underlying
dynamics are obtained. We refer the interested reader to \cite{beguinsurveytoro,PassThesis} and references therein for a more
complete\footnote{These surveys are not completely updated as there has been some fast progress in the recent years.} account on this theory.

\smallskip

\ro{The} \emph{topological entropy} measures how chaotic a prescribed dynamical system is. It measures the rate of exponential growth of different orbits in a dynamical system when observed at a given (arbitrarily small) scale. We shall not provide a \ro{formal} definition of topological entropy here (see e.g. \cite[Chapter 3]{KatokHasselblatt}). The precise \ro{formulation} of this notion is rather technical, but it is unimportant to our paper as our proof of positivity of topological entropy relies on obtaining certain dynamical configurations which are interesting by themselves (and which are known to imply positive topological entropy).

\smallskip

When the dimension of the rotation set equals the dimension of the space where the dynamics acts, there exists a relation between the \emph{geometry and arithmetics} of the rotation set and the topological entropy of the system. For instance, for degree one maps on the circle, the topological entropy is bounded from below by an explicit (and optimal) function of the extremal points of the rotation set as shown in \cite{circulo}. In the toral case, the quantity considered for such a lower bound is less explicit and\ro{,} as far as the authors are aware, not optimal. See \cite{llibre,kwapitz,Forcing}.

\smallskip

In the annulus $\A=\mathbb{S}^1\times\R$, a large rotation set is not necessarily associated with large entropy. Integrable twist map, e.g. maps of the form $(x,y)\to (x+ r(y) \ ( mod \ 1),y),$ preserving a foliation by essential circles\ro{,} have zero entropy but may have rotation sets of arbitrarily large size. One can look at the rotation set restricted to certain invariant regions of the annulus and hope to draw better conclusions.

For this \ro{purpose}, the class of invariant sets which turns \ro{out to be} interesting to observe are the essential annular \ro{continua}:
\ro{a continuum $K\subset \A$ is called an \emph{essential annular continuum} if $\A\setminus K$ has exactly two connected components and both of
them are unbounded (and hence $K$ must disconnect both ends of $\A$)}. These \ro{sets are} natural objects in surface dynamics which \az{model}
for instance the mentioned attractors, and have been \ro{the focus} of several works in the field. The topology \ro{of essential annular continua}
can be very simple, as for the circle or the closed annulus itself, and very complex, as it is the case of indecomposable annular continua, for instance, the pseudo-circle.

\smallskip

We mentioned above that for the case where $K\subset\A$ is a closed essential annulus, there is no relation between the length of the rotation interval and the topological entropy. As a next step, one can look at those annular continua containing no essential annulus. For this class of continua, there exists an interesting example by Walker \cite{Walker}, in which an invariant annular continua having empty interior $K$ is constructed having zero entropy and arbitrary large rotation set. Nevertheless, this continuum contains an essential circle inside, that is, $K$ is not irreducible. Irreducible annular continua (see Section \ref{ss.continua}), often called \emph{circloids}, with non-trivial rotation sets are known as interesting examples, and it is possible to construct them so that they are robust in the $C^0$ topology (see \cite{boronski,lecalvez}). Further, as we mentioned before, this kind of dynamics occur as global attractors of dissipative twist maps given by the so called \emph{Birkhoff attractors} \
\cite{lecalvez}, and \ve{are the canonical model for the \emph{strange attractors} of annular diffeomorphisms}.

\smallskip

In this article we show the following complementary facts. For an orientation preserving homeomorphisms $f$ and an attracting invariant circloid $\C$:

\ve{
\begin{itemize}

\item We show in Theorem A that if $\cC$ has a non-trivial rotation set, then some power of $f$ has a topological
horseshoe with a non-trivial rotation set (see Section \ref{s.precise} for the definition of \emph{rotational horseshoe}). Moreover,
this situation is $C^0$-robust, that is, any homeomorphism $C^0$-close to $f$ has a \emph{rotational horseshoe}.

\item In Theorem B we show that there is no relation between the entropy and the length of the rotation set, so the power of $f$ needed in order to find the horseshoe in Theorem A can be arbitrary large for a prescribed rotation set.

\end{itemize}

\smallskip

The first result answers positively (assuming the circloid is a global attractor) a folklore problem about the relation between entropy and rotation
intervals on circloids (see for instance \cite{kororealizacion} and Question 3 in \cite{boronski}). Moreover, the result shows that these kind of
attractors are associated to $C^0$-robust topological horseshoes with rotational information (see the definition below).
The second result is quite surprising: one could expect that the size of the rotation set could impose a lower bound on the topological entropy as it is
the case for degree one maps of the circle.}

\smallskip

The techniques in the proofs allow us to deal with \ro{the related} class of Birkhoff attractors (see definition below).

\smallskip

Next, we give precise statement of the results.

\subsection{Precise statements}\label{s.precise}

In what follows we list the obtained results. Recall that $\A$ stands for the infinite annulus $\A=\mathbb{S}^1\times \R.$ We denote by $\textrm{Homeo}_+(\A)$
the set of homeomorphisms of the annulus which preserve orientation. 

\ve{Given a homeomorphism $f:\mathcal{X}\to\mathcal{X}$
and a partition by $m>1$ elements $R_0,\dots,R_{m-1}$ of $\mathcal{X}$, the \emph{itinerary} function $\xi:\mathcal{X}\to\{0,\dots,m-1\}^{\Z}:=\Sigma_m$
is defined by $\xi(x)(j)=k$ iff $f^j(x)\in R_k$ for every $j\in\Z$.

We say that a compact invariant set $\Lambda\subset\A$ of $f\in\textrm{Homeo}_+(\A)$ is a \emph{rotational horseshoe} if it admits a finite partition
$\mathcal{P}=\{R_0,\dots,R_{m-1}\}$ with $R_i$ open sets of $\Lambda$ so that

\begin{enumerate}

\item the itinerary $\xi$ defines a semiconjugacy between $f|_{\Lambda}$ and the full-shift $\sigma:\Sigma_m\to\Sigma_m$, that is $\xi\circ f=\sigma\circ\xi$ with $\xi$ continuous and onto;

\item for any lift $F$ of $f$, there exist a positive constant $\kappa$ and $m$ vectors $v_0,\dots,v_{m-1}\in \Z\times\{0\}$ so that
$$\|(F^n(x)-x)-\sum_{i=0}^{n}v_{\xi(x)} \|<\kappa\mbox{ for every }x\in\pi^{-1}(\Lambda),n\in\N.$$

\end{enumerate}

Clearly, the existence of a rotational horseshoe for a map implies positive topological entropy larger
than $\mathrm{log}(m)\geq\mathrm{log}(2)$. Other interesting implications can be obtained; for instance, the
realisation\footnote{A periodic point $x$ \emph{realises} a rational rotation vector $\frac{p}{q}$ (with $p \in \mathbb{Z}^2$ and $q \in \mathbb{Z}_{>0}$ if there is a lift $\tilde x$ of $x$ so that $F^q(\tilde x)= \tilde x + p$.} of every rational rotation vector in $\rho_{\Lambda}(F)$.

\begin{thmA}\label{teo.positivo}

Assume that $f\in\textrm{Homeo}_+(\A)$ has a global attractor $\C$ given by a circloid for which $\rho_{\C}(F)$ is a non-trivial interval,
where $F$ is a lift of $f$. Then, there exists $n_0$ so that $f^{n_0}$ has a rotational horseshoe $\Lambda$ contained in $\C$.
Moreover, there exists a $C^0$-neighborhood $\mathcal{N}$ in $\textrm{Homeo}_+(\A)$ of $f$ so that for every $g\in\mathcal{N}$ we have a rotational horseshoe $\Lambda_g$ for $g^{n_0}$.
In particular, $h_{\textrm{top}}(g)>\varepsilon_0$ for all $g\in\mathcal{N}$ and some positive constant $\varepsilon_0$.

\end{thmA}

This result and Theorem C below can be derived from a more general statement given by Theorem \ref{t.generalA} in Section \ref{ss.teoC}.}





The complementary result is given by the following.

\begin{thmB}\label{teo.ejemplos}

Given $\varepsilon>0$ there exists a smooth \ro{diffeomorphism} $f\in\textrm{Homeo}_+(\A)$ admitting a global attractor $\cC$ which is a circloid, such that $\rho_{\cC}(F) \supset [0,1]$ for some lift $F$ of $f$, while $h_{\textrm{top}}(f) < \varepsilon$.

\end{thmB}

\ve{As we mentioned above, this implies that for a prescribed positive length of the rotation interval, the minimum positive integer $n_0$ as in theorem A (for which $f^{n_0}$ has a rotational horseshoe)
could be arbitrary large.}


Recall that \ro{given a riemannian manifold $M$} a diffeomorphism $f:M\to M$ is said to be \emph{dissipative} whenever there exists $\eps>0$ such that  $|\textrm{det}(Df_x)|<1-\eps$
for every $x\in M$. Further, recall that a diffeomorphism $f:\A\to\A$ is said to be a \emph{twist map}  if for some lift $F$ of $f$ there is $\eps>0$ so that
$DF_x((0,1))=(a(x),b(x))$ with $ \eps< a(x) < \frac{1}{\eps}$.

Given a dissipative twist map of the annulus which maps an essential closed annulus into its interior one can associate a global attractor $\Lambda$, given by the intersection of the
iterates of the annulus. This is an annular continua with empty interior and contains a unique circloid $\C$ which is the so called \textrm{Birkhoff attractor} (see \cite{lecalvez}).
Notice however, that this situation differs from the situation in Theorem A, as the Birkhoff attractor $\C$ might not be an attractor in the usual sense. In other words, it could be
the case that $\Lambda\neq\C$. \ve{In this setting, we show the following result.

\begin{thmC}\label{entropiaenatractorBir}

Assume that $f:\A\to\A$ is an orientation preserving diffeomorphism, which is dissipative, verifies the twist condition and
$f(\mathcal{A})\subset\mathcal{A}$ for some compact essential annulus $\mathcal{A}\subset\A$. Further, assume that $\rho_{\mathcal{C}}(F)$
is a non-trivial interval, were $\C$ is the Birkhoff attractor of $f$. Then, there exists $n_0$ so that $f^{n_0}$ has a rotational horseshoe $\Lambda$.
Moreover, there exists a $C^0$-neighborhood $\mathcal{N}$ of $f$ in $\textrm{Homeo}_+(\A)$ so that for every $g\in\mathcal{N}$ we have a rotational horseshoe $\Lambda_g$ for $g^{n_0}$.
In particular, $h_{\textrm{top}}(g)>\varepsilon_0$ for all $g\in\mathcal{N}$ and some positive constant $\varepsilon_0$.

\end{thmC}

}



We finish adapting the proof of Theorem B to show that the topological entropy and the length of rotation intervals are again not related for Birkhoff attractors. The difference with Theorem B is that although in this case we have dissipation, we can not ensure that the global attracting set coincides with the unique invariant circloid it contains.

\begin{thmD}

For every $\varepsilon>0$ there exists a dissipative twist smooth diffeomorphisms $f:\A\to\A$ having an Birkhoff attractor $\C$ with $\rho_{\C}(F)\supset[0,1]$ and $h_{\textrm{top}}(f|_{\mathcal{C}})<\varepsilon$.

\end{thmD}

\begin{obs}

There is a certain analogy between Birkhoff attractors and \emph{regions of instability} of conservative annulus homeomorphisms (see for example \cite{frankslecalvez}).
Recall that an \emph{instability region} $R$ for an area-preserving annular homeomorphism is an invariant \ro{compact connected set whose boudary is given by two disjoint essential annular continua $C_-$ and $C_+$,} having a point with $\alpha$-limit in $C_-$ and
$\omega$-limit in $C_+$, and a point with $\omega$-limit in $C_-$ and $\alpha$-limit in $C_+$. In a recent article P. Le Calvez and F. Tal \cite{Forcing} (see also \cite{frankshandel}) have shown that whenever an instability region has a non-trivial interval as rotation set, then the map has positive entropy. In the process of proving Theorem B and D we must construct an \az{instability} region (of a smooth twist map) with rotation set containing $[0,1]$ and arbitrarily small entropy, showing that in this context again, there is no relation between the size of the rotation interval and the topological entropy of the map.
\end{obs}

%
%






\subsection{The techniques.}  We present here some key points in the proofs of Theorem A and B avoiding technicalities.

\smallskip

The main idea behind the proof of Theorem A is to work in the universal cover and use the fact that there are periodic points turning at different speeds in order to construct a
topological rectangle $R$ which has an iterate intersecting itself and a translate of itself as well in a \emph{Markovian} way.
\ve{Using this configuration and the results of \cite{KY} we obtain a rotational horseshoe as defined above.}

We are not able to control the number of iterates we need to obtain this intersection (and it would be impossible in view of Theorem B) but we give some geometric criteria that
forces a lower bound. The construction of this rectangle requires entering into properties
of the topology of non-compactly generated continua (a generalization of indecomposable continua). The two key points are the construction
of ``stable'' sets for periodic points, obtained by approaching the dynamics by hyperbolic dynamics in the $C^0$-topology (this step works in quite large generality,
see Theorem \ref{t.continuum}), and then show that for periodic points having different rotation vectors, these ``stable sets'' intersect both boundaries of a given annulus containing the circloid (Proposition \ref{buenosconexos}).


\smallskip

In order to construct the examples of Theorem B, the idea is to work with $C^1$-perturbations of a twist-map,
which are based on the $C^1$-connecting lemma for pseudo-orbits in the conservative setting, due to M.C. Arnaud, C. Bonatti and S. Crovisier (\cite{ArnaudBC}). The use of this \ro{theorem} in this case is not completely straight-forward, as it is a result of generic nature, and we need to take care of some non-generic properties of our examples. However, by an inspection of the \ro{proof in} \cite{CrovisierPanoramas}, one can state a suitable version in order to obtain our desired perturbations. We remark that similar kind of \ro{perturbative techniques} were already considered in \cite{girad} for different purposes. Using these perturbations one can construct a smooth diffeomorphism of the closed annulus which is conservative and for which points in each of the boundary components are homoclinically related (and have different rotation numbers). A further perturbation \ro{allows us to destroy} the annulus and an attracting circloid emerges, which still has the same rotation set. As the derivative of the
original
map had small growth, the same holds for the perturbations which ensures small entropy.

\smallskip

Theorem B shows that the usual arguments dealing with
Nielsen-Thurston theory as used for instance in \cite{llibre} and \cite{kwapitz} do not work for proving Theorem A. On the other hand, recently Le Calvez and Tal \cite{Forcing} developed a \emph{forcing} technique based in Le Calvez's foliation by Brower lines (\cite{lecalvez-fol}), which could provide an alternative proof of the positive entropy in Theorem A.


\smallskip

Let us end this introduction by mentioning that Crovisier, Kocsard, Koropecki and Pujals have announced progress in the study of a particular family of diffeomorphisms of the annulus which they call \emph{strongly dissipative}. In this class, they are able, among other things, to prove positive entropy if there are two rotation vectors and the maximal invariant set is transitive.  We notice that even if our proof does not give lower bounds on the entropy in all generality (and it cannot give one because of Theorem B), it is possible that for some families such a lower bound exists. In particular, we re-emphasize that our method does give a lower bound after some configuration is attained (see Lemma \ref{l.gromov}). 

\subsection{Organization of the paper.}

The structure of the article is the following. We start with some preliminaries in Section \ref{s.Prelim}. From those, subsections \ref{sectionentropy} and \ref{ss.connecting}
are used in the proof of Theorem B while subsections \ref{ss.continua} and \ref{ss.primeend} are used for the proof of Theorem A.

Theorem A and B have independent proofs and can be read in any order. Theorem A and Theorem C are proved in section \ref{chapterposent}, whereas Theorem B and D are proved in section \ref{chapterperttwist}. \ve{In subsection \ref{ss.teoC} a generalization of Theorem A is obtained from which Theorem C can be derived.}

\subsection{Acknowledgments:} We would like to thank M. C. Arnaud, S. Crovisier, T. J\"ager, A. Koropecki, P. Le Calvez and F. Tal and \ve{the referees} for their input to this paper.


\section{General Preliminaries}\label{s.Prelim}

We introduce in this section some preliminary well known results which will be used later. Some results hold in higher dimensions too but we will always restrict to the surface case. The reader can safely skip this section and come back when results are referenced to.

\subsection{A remark on continuity of entropy in the $C^1$-topology}\label{sectionentropy}

For a $C^1$-surface map $f:M\to M$ there is a bound on the topological entropy given by
$$ h_{top}(f) \leq  2 \log\  \textrm{sup}_{x\in M}  \| Df_x\|\ = 2 \log \| Df \|\ .$$

See for example \cite[Corollary 3.2.10]{KatokHasselblatt}. Since $h_{\textrm{top}}(f)=\frac{1}{n}h_{\textrm{top}}(f^n)$, we have

$$ h_{\textrm{top}}(f)\leq \frac{2}{n}\log \| Df^n\|\mbox{ for all }n\in\N\ .$$

We deduce the following.

\begin{prop}\label{p.continuityentropy} Let $f: M \to M$ be a $C^1$-surface map such that $\lim_{n\to \infty} \frac{2}{n}\log \|Df^n\| =0$. Then, for every $\eps>0$ there exists a $C^1$-neighborhood $\mathcal{N}$ of $f$ such that if $g \in\mathcal{N}$ one has that $h_{top}(g)< \eps$.
\end{prop}

\begin{proof}
Fix $\eps>0$ and choose $n>0$ such that $\frac{2}{n} \log \|Df^n\| < \eps$. Choose a $C^1$-\az{neighbourhood} $\mathcal{N}$ of $f$ so that for every $g\in \mathcal{N}$ one has $\frac{2}{n} \log \|Dg^n\| < \eps$. By the estimate above, it follows that for every $g\in \mathcal{N}$ one has that $h_{top}(g) < \eps$.
\end{proof}

%
%
%

\subsection{Connecting lemma for pseudo-orbits}\label{ss.connecting}

In this section we state a $C^1$-perturbation lemma for pseudo-orbits in the conservative setting in the spirit of the well known pseudo-orbit connecting lemma (\cite{BonattiCrovisier,ArnaudBC}). 

\medskip

Let $M$ be a surface, $\nu$ an area form in $M$ and let $\textrm{Diff}^{\ 1}_{\nu}(M)$ be the space of $C^1$ area preserving diffeomorphisms, with the $C^1$ topology. We recall that given $\varepsilon$, a finite sequence $(z_k)_{k=0}^n$ is a $\varepsilon$-\emph{pseudo-orbit} (or $\eps$-\emph{chain}) from $p\in M$ to $q\in M$ when $z_0=p,\ z_n=q$ and
$$d(f(z_k),z_{k+1})<\varepsilon,\mbox{ for all }k=0,\dots,n-1\ .$$

Consider a compact set $K\subset M$. For $x,\ y\in M$ we denote $x\dashv_K y$ if for every $\varepsilon>0$ there exists a $\varepsilon$-pseudo-orbit $(z_k)_{k=0}^{n}$ with $z_0=x$, $z_n=y$ and
$$ f(z_k),\ z_{k+1}\in K \mbox{ whenever } f(z_k)\neq z_{k+1}\ .$$

Denote by $\textrm{Diff}_{\nu,per}^{\ 1}(M)$ the set of those $f \in \textrm{Diff}_{\nu}^{\ 1}(M)$ for which the set
of periodic points of period $k$ is finite, for all $k\in\N$. Recall that the \emph{support} of a perturbation $g$ of $f$ is the set of points $x\in M$ where $g(x)\neq f(x)$.

\begin{thm}[A version of the $C^1$-connecting lemma for pseudo-orbits \cite{CrovisierPanoramas}]\label{connecting2}
Let $M$ be a compact surface possibly with boundary and $f\in\textrm{Diff}_{\nu,per}^{\ 1}(M)$. Given a \az{neighbourhood} $\mathcal{N}\subset \textrm{Diff}_{\nu}^{\ 1}$ of $f$, there exists $N=N(f,\mathcal{N})$ such that:
\begin{itemize}
\item if $K$ is a compact set disjoint from the boundary,
\item $U$ is an arbitrary small neighborhood of $K\cup \dots\cup f^{N-1}(K)$,
\item and $p,\ q\in M$ with $p\dashv_K q$,
\end{itemize}
then, there exist a perturbation $g\in\mathcal{N}$ of $f$ supported in $U$ and $n>0$ such that $g^n(p)=q$.
\end{thm}

This result follows with the same proof of Theorem III.1 presented in \cite{CrovisierPanoramas} via \cite[Theorem III.4]{CrovisierPanoramas} where the choice of $N$ appears. The difference is that in \cite{CrovisierPanoramas} the
statement requires the complete pseudo-orbit to be contained in $K$ while here we demand only the \emph{jumps} to be contained there. By an inspection of the proofs in \cite{CrovisierPanoramas} one can see that the perturbations 
are only performed when the pseudo-orbit has jumps, so our statement holds with only minor modifications.

\begin{obs}\label{remark:smooth} The diffeomorphism $g$ can be considered to be as smooth as $f$ since it is obtained by composing a finite number of elementary perturbations with small support, all of which are smooth (though their $C^r$-size with $r>1$ might be \ro{large}).
\end{obs}


\subsection{Some properties of separating continua}\label{ss.continua}

We first recall some basic facts about continua and separation properties in surfaces. We refer the reader to \cite{BG} and references therein for more information. After this, we will show a property of irreducible annular continua that will be useful in the proof of Theorem A.

\medskip

Throughout this article we consider the annulus $\A=\mathbb{S}^1\times\R$ and $\pi:\R^2\to \A$ the usual covering map. Further, we will fix a  two-point compactification of $\A$ given by the sphere $S^2$ and two different points $+\infty,-\infty\in S^2$.

Recall that a continuum is a compact \ro{non-empty} connected metric space. We say a continuum $E\subset\A$ is \emph{essential}, whenever there are two unbounded connected components in $\A\setminus E$. These connected components are denoted in general by $\mathcal{U}^+$ and $\mathcal{U}^-$ where the first one accumulates in $+\infty$ and the second one in $-\infty$, when considered in $S^2$. Notice that there could be also several bounded connected components in $E^c$.
Non essential \ro{continua} in $\A$ are called \emph{inessential}, and can be \az{characterised} as those continua contained in some topological disk in $\A$.


An annular continuum $K\subset \A$ is an essential continuum so that $K^c$ contains no bounded connected components. Finally, an \emph{irreducible annular continuum} or \emph{circloid} $\C$, is an annular continuum which does not contain properly any other annular continua. As it is well know, the topology of these continua can be very simple as the one of the circle, or extremely complicated as the case of the pseudo-circle. It can be the case where the circloid has non-empty interior; an example (and figure) can be found, for instance, in  \cite{ratajula}.  


When a circloid has empty interior it is called \emph{cofrontier} as it coincides with the boundaries  of $\mathcal{U}^+$ and $\mathcal{U}^-$. A \ro{partial} converse
result holds: whenever an annular continuum $\C$ verifies that $\partial\C=\partial \mathcal{U}^+\cap \partial \mathcal{U}^-$, we have that $\C$ is a circloid
(with possible non-empty interior). See \cite[Corollary 3.3]{Jaeger}.

\smallskip

There is an important class of continua, which is associated to a complicated topology, defined as follows. An \emph{indecomposable} continuum $\C$ is a continuum such that whenever $C_1$ and $C_2$ are any pair of continua included in $\C$ with $\C=C_1\cup C_2$, we have that $C_1=\C$ or $C_2=\C$. In particular, one can define
\emph{indecomposable cofrontier}. This definition is not suitable for circloids having non-empty interior, as one can observe that in this case the continua can be always decomposed. Nevertheless, a suitable \az{generalisation} of indecomposabilty for this situation can be considered, given by the following (see \cite{jagerpass}).

\smallskip

Let $\C\subset \A$ be an essential annular continua. We say that $\C$ is \emph{compactly generated} if there exists
a compact connected set $\hat{\C}$ in $\R^2$ so that $\pi(\hat{\C})=\C$ (such continuum $\hat{C}$ is called a \emph{compact generator}).
In particular this definition can be applied to essential circloids. The annular continua which are \emph{not \az{compactly} generated} and \emph{indecomposable} annular continua have
strong relationships even if their properties are slightly different (see e.g. \cite[Remarks 1.1 and 5.5]{jagerpass}).  For this paper, the notion of being non-compactly
generated is the most suitable.

\smallskip

In this article we deal only with \emph{non compactly generated circloids} as compactly generated ones do not support two rotation vectors for a given dynamics. This result was originally proved for Birkhoff attractors in \cite{lecalvez} and then \az{generalised} for cofrontiers in \cite{BG}. Finally, it was extended
in \cite{jagerpass} to deal with circloids. Although in this last reference the proof is not explicitly given for the non-empty interior case, as it is remarked by the authors, the proof they give works exactly as it is written  for circloids with non-empty interior (see  \cite[Remark 5.5]{jagerpass}).

\begin{thm}[\cite{BG,jagerpass}]\label{teo-circindec}  Let $\in\textrm{Homeo}_+(\A)$ having an invariant circloid $\C$ such that $\rho_\C(F)$ contains two different rotation vectors for some lift $F$ of $f$. Then, $\C$ is non
compactly generated.
\end{thm}





\ve{We establish next a proposition concerning the topology of non-compactly generated circloids.
Given a circloid $\C\subset \A$, $x\in\C$, $\tilde{\C}=\pi^{-1}(\C)$ and $\hat{x}$ a lift of $x$ we define
$$\hat{C}_{\hat{x}}=\bigcup_{k\in\N} \textrm{c.c.}_{\hat{x}}[\tilde{\C}\cap\pi_1^{-1}([-k,k])].\ \footnote{We denote for any point $x$ in a topological space $\mathcal{X}$
its connected component by $\textrm{c.c.}_x(\mathcal{X})$.}$$

For indecomposable co-frontiers, these are connected sets which lifts the \emph{composants} (see \cite{HockingYoung}).

\begin{prop}\label{prop-composant}

Let $\C$ be a non compactly generated circloid. Then, $\hat{C}_{\hat{x}}$ is an unbounded connected set which does not contain any point $\hat{x}+j$, $j\in\Z\setminus\{0\}$.


\end{prop}

\begin{proof}

By definition, $\hat{C}_{\hat{x}}$ is an increasing union of compact connected sets $C_k=\textrm{c.c.}_{\hat{x}}[\tilde{\C}\cap\pi_1^{-1}([-k,k])]$ containing $\hat{x}$.
Moreover, as $\tilde{\C}$ is connected and unbounded, one can observe that every connected component of $\tilde{\C}\cap\pi_1^{-1}([-k,k])$ must intersect $\partial\pi_1^{-1}([-k,k])$,
so $C_k$ meets $\partial\pi_1^{-1}([-k,k])$ for every $k\in\N$. This implies that $\hat{C}_{\hat{x}}$ is unbounded and connected.

Assume for a contradiction we have $\hat{x}+j\in\hat{C}_{\hat{x}}$ with  $j\in\Z\setminus\{0\}$ Then we have that both $\hat{x}$ and $\hat{x}+j$ belong to $C_k$ for some $k\in\N$.
Thus, $\pi(C_k)$ is an annular continuum, so it must coincide with $\C$ as it is a circloid. But this imply that $\C$ has a compact generator.

\end{proof}

}

\medskip

We are interested in studying inessential continua intersecting a non compactly generated circloid $\C$, which does no meet one of the unbounded components in the complement of the circloid. Fix a non compactly generated circloid $\C$ and let $K\not\subset\C$ be an inessential continuum in $\A$, so that $K \cap \cU^- =\emptyset$. Everything we show for this situation also holds for the complementary case where $K \cap \cU^+ =\emptyset$\ro{.}

\smallskip

In general for a continuum $C\subset\A$ we say that an injective curve $\gamma:[0,+\infty)\to \A$ lands at $z\in C$ from $+\infty$ if $\gamma(t)\in C^c$ for all $t\neq 0$, $\gamma(0)=z$, and $\lim_{t\to +\infty}\gamma(t)=+\infty$ when viewed in $S^2$. When $C$ is an essential continua, the points $z$ \ro{which admit a curve landing on them}, are called \emph{accesible points}, and it is easy to prove that they form a dense set in $C\cap\partial\mathcal{U}^+$. Thus in our situation we can consider a curve $\gamma$ as before, so that

\begin{itemize}

\item $\gamma\cap K=\emptyset$\footnote{We here abuse notation by identifying the curve with its image using the same name.},

\item $\gamma$ lands at $z\in\C$.

\item \ve{$\pi_1(\hat{\gamma})$ is a bounded set for any lift $\hat{\gamma}$ of $\gamma$.}

\end{itemize}

Let $\hat A$ be a connected component of $\pi^{-1}(\cU^+ \setminus \gamma)$. Our main goal is to show the following property which is important to prove Theorem A.

\begin{prop}\label{prop-levacotado}

It holds that $\pi^{-1}(K) \cap \hat A$ is bounded.

\end{prop}

Consider $\tilde{\C}=\pi^{-1}(\C)$, $\tilde{\mathcal{U}}^+=\pi^{-1}(\mathcal{U}^+)$ and $\tilde{\mathcal{U}}^-=\pi^{-1}(\mathcal{U}^-)$. Fix a lift $\hat{K}$ of $K$ which intersects $\hat{A}$. In order to prove Proposition \ref{prop-levacotado}, it is enough to show that only finitely many horizontal integer translations of $\hat{A}$ meets $\hat{K}$.

\smallskip

We prove the following lemma. Recall that $z \in \C$ is the landing point of the curve $\gamma$.

\begin{lema}\label{lema-cuantacomposent} Fix $\hat{z}\in \pi^{-1}(z)$. If $\hat K$ intersects $ \hat C_{\hat z} + k$ and $\hat C_{\hat z} + k'$ then $|k-k'| \leq 1$.
\end{lema}

\begin{proof} Assume otherwise. Without loss of generality we can assume that $k' >k$.

\ve{By the definition of $\hat{C}_{\hat{z}}$ we can consider continua $\Lambda_k \en \hat C_{\hat z} + k$ containing $\hat z + k$
and intersecting $\hat K$, and $\Lambda_{k'} \en \hat C_{\hat z} + k'$ containing $\hat z + k'$
and intersecting $\hat K$. Furthermore, as $\C$ is not compactly generated, Proposition \ref{prop-composant} implies that non of them contain $\hat z+k+1$.}

\smallskip

Let $\hat{\gamma}$ be the lift of $\gamma$ containing $\hat{z}$. We have that $\Lambda_k \cap (\hat \gamma + k) = \{\hat z+k\}$ and $\Lambda_k \cap (\hat \gamma + j) = \emptyset$ for every $j\in\Z\setminus\{k\}$, and the symmetric conditions hold for $\Lambda_{k'}$. See Figure \ref{figtopcirc}.

\tiny

\begin{figure}[ht]\begin{center}
 \psfrag{hatzk}{$\hat{z}+k$}\psfrag{hatzk+1}{$\hat{z}+k+1$}\psfrag{hatzk'}{$\hat{z}+k'$}
 \psfrag{hatgammak}{$\hat{\gamma}+k$}\psfrag{hatgammak+1}{$\hat{\gamma}+k+1$}\psfrag{hatgammak'}{$\hat{\gamma}+k'$}
 \psfrag{Lambdak}{$\Lambda_k$}\psfrag{Lambdak'}{$\Lambda_{k'}$}\psfrag{hatK}{$\hat{K}$}
 \psfrag{H}{$H$}

\centerline{\includegraphics[height=7cm]{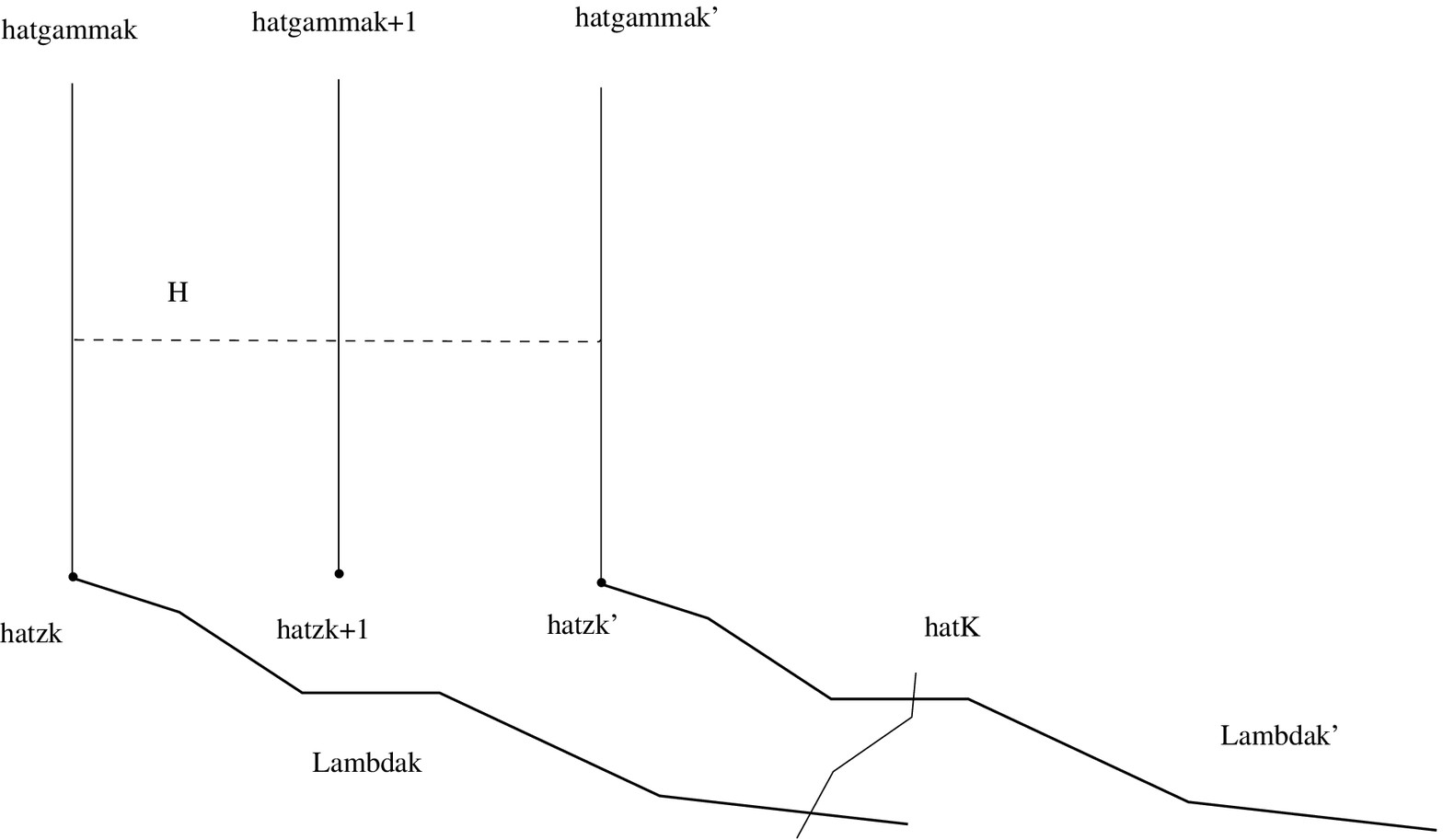}}
\caption{Proof of Lemma \ref{lema-cuantacomposent}.}\label{figtopcirc}
\end{center}\end{figure}

\normalsize

Let $\Gamma= (\hat \gamma + k) \cup \Lambda_k \cup (\hat \gamma + k') \cup \Lambda_{k'}\cup\hat{K}$, which is a closed and connected set. Further, consider
an horizontal segment $H\subset\tilde{\mathcal{U}}^+$ whose endpoints are contained one in $\hat{\gamma}+k$, the other one in $\hat{\gamma}+k'$, and there
are no other intersection between $H$ and $\Gamma$. Notice that this can be easily constructed since the vertical coordinate of points in $\tilde{\C}$ are uniformly bounded.

\smallskip

As $\Gamma\cap\tilde{\mathcal{U}}^-=\emptyset$, we have that $\tilde{\mathcal{U}}^-$ is contained in one connected component of $\Gamma^c$, that we call $U^-$.
Moreover, $H$ must be contained in a different connected component of $\Gamma^c$, as any curve from $H$ to $-\infty$ which does not intersect $\Gamma$,
would allow to separate $\Gamma$ into two connected components, one containing $\hat{\gamma}+k$ and another one containing $\hat{\gamma}+k'$.
We call the connected component of $\Gamma^c$ containing $H$ in its closure by $U^+$.

Due to our assumption, we have that $\hat{\gamma}+k+1$ intersects $H$. Therefore, $\hat z + k +1$ is in the interior of $U^+$ and therefore is not accumulated by $\tilde \cU^-$ which contradicts that $\tilde \C$ is the lift of a circloid.

\end{proof}

Now we are ready to prove Proposition \ref{prop-levacotado}.

\begin{proof}[Proof of Proposition \ref{prop-levacotado}]  Working with $\hat{\gamma}$ as before, we can assume without loss of generality that the closure of $\hat A$ contains both $\hat \gamma$ and $\hat \gamma +1$.

\smallskip

We will show that if a connected component $\hat K$ of $\pi^{-1}(K)$ intersects $\hat A + k$ and $\hat A + k'$ for some $k \neq k'$ then it must intersect either $\hat C_{\hat z} +k$ or $\hat C_{\hat z} + k +1$. Thus, Lemma \ref{lema-cuantacomposent} implies that $\hat{K}$ meets only finitely many lifts of $\hat{A}$ (in fact, at most three consecutive lifts), which implies that $\pi^{-1}(K)\cap\hat{A}$ is bounded.

\smallskip

Without loss of generality, we assume that $\hat K$ intersects $\hat A$ and $\hat A + k$ for some $k \neq 0$ and assume by contradiction that $\hat K$ does not intersect $\hat C_{\hat z}$ nor $\hat C_{\hat z} +1$. Choose $\eta$ a curve contained in $\hat A +k$ landing at point $y \in  \hat K$  (recall that $\hat{K} \cap \gamma =\emptyset$). Choose also a point $x\in \hat K \cap \hat A$.

\smallskip

Recall that we denote by $\pi_1: \R^2 \to \R$ the projection onto the first coordinate. By Proposition \ref{prop-composant} one has that $\pi_1(\hat C_{\hat z})$ is unbounded, and we assume without loss of generality that it  \ro{has no upper bound.} 
Notice that $\pi_1$ is bounded both on $\hat \gamma$ and $\eta$.

\smallskip

Choose a very large $r>0$ and consider a vertical line $v_r=\pi_1^{-1}(r)$ which intersects $\hat C_{\hat z}$ and $\hat C_{\hat z} + 1$ in points $w_0^r$ and $w_1^r$ respectively. Choose a non-separating continua $\Lambda_0^r$ in $\hat C_{\hat z}$ containing $\hat z$ and $w_0^r$ and similarly consider $\Lambda_1^r \en \hat C_{\hat z} +1$ containing $\hat z+1$ and $w_1^r$, which can be done due to the arguments we did before. Define $I_r\subset v_r$ as the segment joining
$w_0^r$ with $w_1^r$.

\smallskip

Let $\Gamma_r = \hat \gamma \cup \Lambda_0^r \cup (\hat  \gamma +1) \cup \Lambda_1^r \cup I_r$. Then, by the same argument we did before, one can consider $U^+(r)$ as the connected component of $\Gamma_r^c$ containing an horizontal segment $H$ joining $\hat{\gamma}$ and $\hat{\gamma}+1$. \ro{One can observe that by construction $\bigcup_{r>0}U^+(r)\supset \hat{A}$, as every point $u$ in $\hat{A}$ can be joined to a point in $H$, with a compact
arc $J$, so that for $r_u$ large enough we have that $J\cap\textrm{pr}_1^{-1}([r_u,+\infty))=\emptyset$, so $J\subset U^+(r)$
for $r\geq r_u$.
Thus, for every $r$ big enough, we find a point of $\hat{K}\cap\hat{A}$ in $U^+(r)$}. Therefore, as $\hat{K}$ is compact and connected, there exits $r_0\in\R$, so that $\hat{K}\subset U^+(r)$ for every $r>r_0$. Notice, that we do not claim
that $U^+(r)\subset\hat{A}$, which is false in general.

\smallskip

On the other hand, as $\eta$ can not intersect $H$, one can see that $\eta$ meets $U^+(r)^c$ for all $r>r_0$. Thus,
if one considers $r'>r_0$ so that $\eta\cap v_{r'}=\emptyset$ (which can be done as $\pi_1(\eta)$ is bounded), we have that $U^+(r')\cap \eta=\emptyset$, otherwise $\eta$ intersects $\Gamma_{r'}\setminus v_{r'}$, which is imposible by construction. This is a contradiction as $\eta\cap\hat{K}\neq\emptyset$ and $\hat{K}\subset U^+(r')$.

\end{proof}

%


\subsection{Prime-end compactification}\label{ss.primeend}

Consider a homeomorphism $f:\A \to \A$ which we can compactify to a homeomorphism $\hat f: S^2 \to S^2$ by adding two fixed points at infinity. In our context, there is a global attractor $\cC$ in $\A$ which is a circloid, this implies that the points at infinity are sources for $\hat f$ and the boundary of their basins coincide with $\partial \C$.

Let $\cU^+$ and $\cU^-$ be the connected components of $\A \setminus \cC$ (which are unbounded).  Denote as $\tilde \cU^\pm$ their lifts to the universal cover $\R^2$ which are connected sets. Let $F: \R^2 \to \R^2$ be a lift of $f$ to $\R^2$, it follows that $F \circ T = T \circ F$ where $T$ is any integer translation in the first coordinate.

We denote as $\hat \cU^\pm = \cU^\pm \cup \{\pm \infty\}$ the corresponding components in $S^2$. These are $f$ (resp. $\hat f$) invariant simply connected open sets and the dynamics \ro{coincides with that of} the basin of a source in each $\hat \cU^\pm$. We introduce here some very basic facts from prime-end theory used in this paper and refer to the reader to \cite{Mather,KLCN,Matsumotoprimeends} or \cite[Section 2.2]{kororealizacion} for more details and references.

The prime end compactification of $\hat \cU^\pm$ is a closed topological disk $U^\pm \cong \mathbb{D}^2$ obtained as a disjoint union of $\hat \cU^+$ and a circle with an appropriate topology (see \cite{Mather}).

If one lifts the inclusion $\cU^\pm \hookrightarrow \hat \cU^\pm \setminus \{\pm \infty \}$ one obtains a homeomorphism $p^\pm : \tilde \cU^\pm \to \HH^2$, and by considering $\hat F^\pm$ the homeomorphism of $\HH^2$ induced by $F$ on $\tilde \cU^{\pm}$  (i.e. such that $p^\pm \circ F = \hat F^{\pm} \circ p^{\pm}$) one sees that $\hat F^\pm$ extends to a homeomorphism of the closure $\textrm{cl}[{\HH^2}]$ in $\R^2$ and still commutes with horizontal integer translations. This allows one to compute the \emph{upper and lower prime end rotation numbers} of $\cC$ (see \cite{kororealizacion} for more details). However, we shall not use this, but just use the following facts about $\hat F^\pm$ and its relation with $F$.

\begin{itemize}

\item The map $\hat F^\pm$ restricted to $\partial \HH^2 \cong \R$ is the lift of a circle homeomorphism where the horizontal integer translations act as deck transformations.

\end{itemize}

 We finish with a last topological property for the Prime-end compactification. Let $\mathcal{U}$ be a topological disk bounded by a continuum $\mathcal{C}$ contained in some surface. For any curve $\gamma:[0,1]\to\mathcal{U}\cup\mathcal{C}$, with $\gamma(t)\in\mathcal{\C}$ iff $t=0$, we have that the corresponding curve $\eta:(0,1]\to\D$ of $\gamma|_{(0,1]}$ admits a unique continuous extension to a curve $\overline{\eta}:[0,1]\to\D$, with $\overline \eta(0)\in\partial\D$.


\section{Attracting circloids and entropy}\label{chapterposent}

In this section we give a proof of Theorem A, stating that an attracting circloid with two different rotation numbers \ve{for a map $f$ has a rotational horseshoe associated to some power $f^{n_0}$}. We first present a proof Theorem A. Then, in Section \ref{ss.teoC} we show how the hypothesis in Theorem A can be relaxed to obtain
a more general statement, see Theorem \ref{t.generalA}, from which we can obtain Theorem C.

\smallskip

To fix the context, we introduce the following hypothesis:

\begin{itemize}

\item[($\mathrm{GA}$)] $f: \A \to \A$ is an orientation preserving homeomorphism of the infinite annulus $\A = \mathbb{S}^1 \times \R$ such that it has a global attractor $\mathcal{C}$ which is a circloid and the rotation set of $f$ restricted to $\mathcal{C}$ is a non-trivial interval.

\end{itemize}

Theorem A states that if $f$ verifies ($\mathrm{GA}$) then \ve{there is a rotational horseshoe for some power
$f^{n_0}$}. Notice that by Theorem \ref{teo-circindec}, property ($\mathrm{GA}$) implies that the circloid $\mathcal{C}$ must be non compactly generated.

\subsection{Some previous definitions}

Chose $\mathcal{A}\subset \A$ any annular \az{neighbourhood} of $\cC$ (i.e. homeomorphic to $\mathbb{S}^1\times [-1,1]$ containing $\cC$ in its interior) so that $f(\mathcal{A})\subset\mathcal{A}$. Since $\cC$ is a global attractor, we have $\C=\bigcap_{n\in\N}f^n(\mathcal{A})$.

Denote by $\mathcal{U}^+$ and $\mathcal{U}^-$ the connected components of $\A \setminus \C$ whose projections into the second coordinate is not bounded from above (resp. below) and by $\partial^+ \mathcal{A}$ and $\partial^-\mathcal{A}$ the connected components of $\partial \mathcal{A}$, contained in
$\mathcal{U}^+$ and $\mathcal{U}^-$ respectively.

\smallskip

Given any essential annulus $\mathcal{A}$ in $\A$, with boundary components $\partial^- \mathcal{A}$ and $\partial^+ \mathcal{A}$, we say that a continuum $D$  \emph{joins the boundaries of} $\cA$ if it verifies the following conditions:

\begin{enumerate}
\item $D\subset \mathcal{A}$ and it intersects both boundaries, i.e. $D\cap\partial^+ \mathcal{A}\neq\emptyset$, $D\cap\partial^- \mathcal{A}\neq\emptyset$.

\item $D$ is inessential (i.e, it is contained in a topological disk).
\end{enumerate}

Let $D_0$ and $D_1$ be two disjoint continua in $\cA$ joining the boundaries. It follows that $\cA \setminus (D_0 \cup D_1)$ has at least one connected component $R$ which contains a curve joining the boundaries of $\cA$. Such a component must verify that its closure intersects both $D_0$ and $D_1$ and it will be called a \emph{rectangle adapted} to $D_0$ and $D_1$. It is easy to show that it is an open connected subset of $\cA$ whose boundary (relative to $\cA$) is contained in $D_0 \cup D_1$.

\smallskip

Recall that we have considered $\pi: \R \times \R \to \A=\mathbb{S}^1 \times \R$ the canonical projection where $\mathbb{S}^1$ is identified with $\R/_\Z$. Given an inessential
continuum $D \en \cA$ which joins the boundaries of $\mathcal{A}$, one considers $\hat D$ to be a connected component
\footnote{Notice that since $\cA$ is essential, one has that $\pi^{-1}(\cA)$ is connected.} of $\pi^{-1}(D)$ in $\hat \cA=\pi^{-1}(\cA)$.
One defines \emph{the right} of $\hat D$ to be the (unique) unbounded component of $\hat \cA \setminus \hat D$ accumulating in $+\infty$ in the first coordinate.
One defines \emph{the left} of $\hat D$ symmetrically.

Notice that if $D_0$ and $D_1$ are two disjoint continua joining the boundaries of $\cA$, and $R$ is a rectangle adapted to $D_0$ and $D_1$ then, if $\hat R$ is a connected
component of the lift of $R$, there is a unique connected component of the lift of $D_0$ (resp. $D_1$) such that it intersects the closure of $\hat R$.
Call these components by $\hat D_0$ and $\hat D_1$.



\ve{

\subsection{A criteria for producing rotational horseshoes}

We start with a lemma which guarantees the existence of a rotational horseshoes. Then we prove that under the
hypothesis of Theorem A, we can apply this result. The proof of the lemma is given by the well known
construction of the \emph{Smale's horseshoe}, which is generalised in \cite{KY}.

\begin{lemma}\label{l.gromov}
Let $\mathcal{A}\subset \A$ be an essential annulus as before, and $h:\mathcal{A} \to \mathcal{A}$ be a continuous map with
$h(\mathcal{A})\subset\textrm{int}(\mathcal{A})$. Denote by $\tilde h$ the lift of $h$ to the universal cover. Assume we have two disjoint continua
$D_0$ and $D_1$ joining the boundaries of $\cA$ such that for some rectangle $R$ adapted to $D_0$ and $D_1$ and some connected component
$\hat R$ of the lift of $R$ there is a positive integer $j$ with the following properties:


\begin{enumerate}

\item if $\hat D_0$ and $\hat D_1$ denote the connected components of the lift of $D_0$ and $D_1$ intersecting the closure of $\hat R$ we have that $\tilde h(\hat{D_0})$ is at the left of the closure of $\hat{R}$,

\item $\tilde h(\hat D_1)$ is at the right of the closure of $\hat{R} + j $.

\end{enumerate}

Then, there exists a $C^0$-neighborhood $\mathcal{N}$ of $h$ in $\textrm{Homeo}_+(\A)$ such that every
$g\in \mathcal{N}$ has a rotational horseshoe so that the associated partition has at least $j+1$ symbols.

\end{lemma}

\begin{proof}

The proof is given by a simple inspection of \cite{KY}. The hypothesis we have for $D_0$ and $D_1$ implies
that an adapted rectangle $R$ is under the \emph{horseshoe hypothesis} (together with $\A$ and the map $h$),
so applying Theorem 1 in the quoted article, we already have the existence of a compact $h$-invariant set $\Lambda\subset\A$
for which  the first condition of the definition of rotational horseshoe is verified (See Figure \ref{shift}).
As noticed by the authors, the semiconjugacy is constructed by using a partition $S_0,\dots,S_{m-1},\ m-1\geq j$
by a finite pairwise disjoint compact sets (which relative to $\Lambda$ are open). Moreover, the construction in our particular case implies that this partition
can be considered so that

\begin{enumerate}

\item for every $i=0,\dots,m-1$ we have a lift $\hat{S}_i$ of $S_i$ in $\hat{R}$,

\item $\tilde{h}(\hat{S}_i)\subset \hat{R}+v_i$ for some integer vector $v_i=(l_i,0)$, where $l_0\leq 0$ and $l_{m-1}>j$.

\end{enumerate}

As the semiconjugacy is constructed by the itinerary function $\xi$ associated to $S_0,\dots,S_{m-1},$
one can deduce by a simple induction that given any point $\hat{x}\in\pi^{-1}(\Lambda)\cap\hat{R}$ lift of $x$, we have
$$\| (\tilde{h}^n(\hat{x})-\hat{x})-\sum_{i=0}^{n-1}v_{\xi(x)(i)} \|<\kappa,$$
where $\kappa$ is the diameter of $\hat{R}$. This implies that $\Lambda$ is a rotational horseshoe.

\end{proof}

}

In order to prove Theorem A, the crucial idea is the following: using the fact that the dynamics is given
on a circloid, and that it is an attractor, we will construct a sort of stable manifolds for some periodic points
$p_0$ and $p_1$, given by two continua $C_0$ and $C_1$, so that they have to intersects both components $\mathcal{U}^+$ and $\mathcal{U}^-$.
These continua will play the role of $D_0$ and $D_1$ in the hypothesis of the last lemma, and this will provide the rotational horseshoe. 

\smallskip

We see in the next lema how the existence of the continua as above allow us to use the previous lemma.

\begin{lemma}\label{l.rotationentropy}
Let $f:\A\to \A$ verifies $(\mathrm{GA})$ and assume that there exist two periodic points $p_0,\ p_1$ with different rotation numbers and two contractible continua $C_0,\ C_1$ containing $p_0,\ p_1$ respectively, such that for $i=0,1$
\begin{enumerate}
\item $f^{n_i}(C_i)\subset C_i$ where $n_i$ is the period of $p_i.$
\item $C_i$ is inessential and intersects both boundaries of $\mathcal{A}.$
\end{enumerate}
\ve{Then, $f^{n_0}$ has a rotational horseshoe for some $n_0\in\N$. Moreover, there exists a $C^0$-neighborhood $\mathcal{N}$ of $f$
in $\textrm{Homeo}_+(\A)$ so that for any $h\in\mathcal{N}$ we have that $h^{n_0}$ has a topological horseshoe}.
\end{lemma}

We remark that we are not assuming that the sets $C_i$ are contained in $\mathcal{A}$, so we can not consider them as joining boundary componentes of $\cA$.

\begin{proof} Consider an iterate $g$ of $f$ and a lift $G$ to the the universal cover $\tilde{\mathcal{A}}$  so that both $p_0$ and $p_1$ are fixed and their
lifts $\tilde p_0$ and $\tilde p_1$ verify $G(\tilde p_0)=\tilde p_0-j$ and $G(\tilde p_1)=\tilde p_1+l$ for some positive integers $j,\ l$ (i.e., $p_0$ rotates negatively
and $p_1$ rotates positively).

As $C_i$ and $\mathcal{A}$ are forward invariant by $f^{n_i}$ (and therefore also for $g$) we have for $i=0,1$ that $g(C_i\cap\mathcal{A})\subset C_i\cap \mathcal{A}.$
Further, as $C_i$ intersects both boundaries of $\mathcal{A}$ and $p_i$ are contained in the interior of $\cA$ there exist some continua $D_i\subset C_i$
in $\mathcal{A}$ for $i=0,1$, joining the boundary componentes of $\mathcal{A}$, see figure \ref{shift} (for a proof of this folklore topological fact, see
for instance Theorem 14.3 in \cite{newtopolcont}).

\smallskip

We pick now some rectangle $R$ adapted to $D_0,\ D_1$, and  $\hat R$ a connected component of the lift of $R$. Let
$\hat{C}_0$ and $\hat{C}_1$ be the lifts of $C_0$ and $C_1$, containing $\hat{D}_0$ and $\hat{D}_1$ as defined above.
It is easy to see that the sets $C_0$ and $C_1$ must be disjoint, as they are both forward invariant for $g$ and have different rotation vectors.

As both $\hat{C}_0$ and $\hat{C}_1$ have bounded diameter, and rotate negatively and positively, we must
have for some sufficiently large $n\in\N$ that

\begin{itemize}

\item $G^n(\hat D_0)$ is at the left $\hat R$,

\item $G^n(\hat D_1)$ is at the right of $\hat R+1$

\end{itemize}

Lemma \ref{l.gromov} now implies that \ve{$g^n$ has a rotational horseshoe, so it does a power of $f$. Furthermore,
as the configuration above reminds for small perturbations of $g$, we obtain the the same result in a $C^0$-neighborhood of $f$.}

\end{proof}

\tiny

\begin{figure}[ht]\begin{center}
 \psfrag{C0}{$C_0$}
 \psfrag{D0}{$D_1$}
 \psfrag{C1}{$C_1$}
 \psfrag{D1}{$D_0$}
 \psfrag{R}{$R$}
 \psfrag{Rn}{$R$}
 \psfrag{Rn1}{$R+1$}
 \psfrag{C1k}{$C_1+k$}
 \psfrag{C0k}{$C_0-j$}
 \psfrag{FkR}{$G^n(R)$}
 \centerline{\includegraphics[height=7cm]{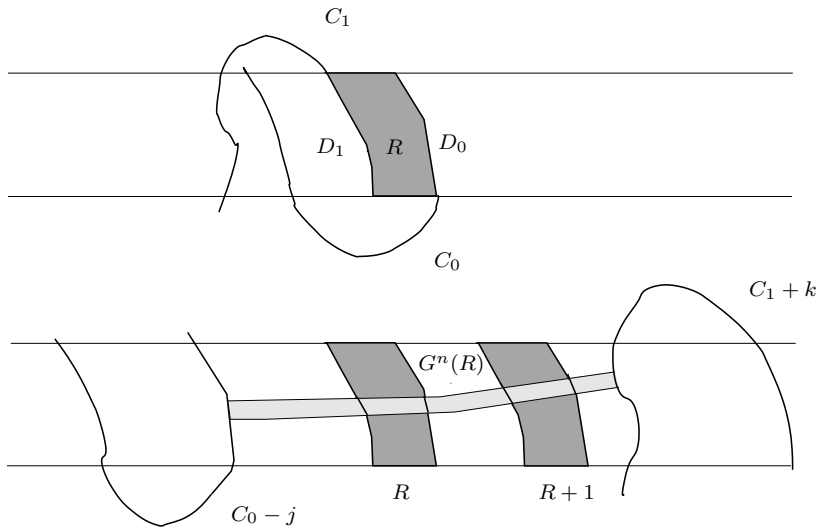}}
\caption{The rotational horseshoe.}\label{shift}
\end{center}\end{figure}

\normalsize

\subsection{A first reduction}

The next result, whose importance we believe transcends the context, will be proved in the next subsection. We will use it here in order to complete the proof of Theorem A.

\begin{thm}\label{t.continuum}
Let $f:\A\to \A$ verifies $(\mathrm{GA})$ and let $p\in\partial \C$ be a periodic point. Then, there exist an \az{inessential} continuum $C_p$ containing $p$ such that $f^{n_p}(C_p)\subset C_p$ where $n_p$ is the period of $p$ and $C_p \cap \partial\mathcal{A}\neq\emptyset.$
\end{thm}

Notice that the continuum $C_p$ might not intersect a priori both boundary components of $\cA$. Moreover, although $C_p$ meets $\C^c$, it may happens that $C_p$ intersects only one of the unbounded connected components $\mathcal{U}^+$ and $\mathcal{U^-}$, that is, $C_p\subset (\mathcal{U}^-)^c$ or $C_p\subset(\mathcal{U}^+)^c$.

\smallskip
We now proceed with the proof of Theorem A assuming Theorem \ref{t.continuum}. By Lemma \ref{l.rotationentropy} it is enough to find two periodic points $p_0$ and $p_1$ with different rotation vectors for which $C_{p_0}$ and $C_{p_1}$ intersect both boundary components of $\mathcal{A}$, \ro{which is equivalent to the following condition since $\C$ is a global attractor:}

\begin{equation}\label{eq:continuossalen}
C_{p_0}\cap \mathcal{U}^+\neq\emptyset\mbox{  and }\ C_{p_0}\cap\mathcal{U}^-\neq\emptyset\ ,\ C_{p_1}\cap \mathcal{U}^+\neq\emptyset\mbox{  and }\ C_{p_1}\cap\mathcal{U}^-\neq\emptyset\ .
\end{equation}

\smallskip
We conclude the section by proving the existence of periodic points $p_0$ and $p_1$ so that \eqref{eq:continuossalen} holds.

\smallskip

Let us state the following realisation theorem of \cite{KoroPass} which improves previous results \cite{kororealizacion} and \cite{BG}.
Here is one of the essential points were we use that $\cC$ is irreducible (see \cite{Walker}). Notice that if one wishes to use \cite{BG} instead of \cite{KoroPass} similar results hold but one needs to add the assumption that the circloid $\cC$ in Theorem A has empty interior. 

\begin{thm}[Theorem G of \cite{KoroPass}]\label{t.realizacion}
Let $h: \A \to \A$ be a homeomorphism of the annulus preserving a circloid $\cC$ so that
 $\rho_{\cC}(H)$ is non-singular for any lift $H$ of $h$. Then, every rational point in the rotation set $\rho_{\cC}(H)$
is \az{realised} by a periodic orbit in $\partial\cC$.
\end{thm}

The idea is to use points which are in $\partial \cC$ but are not accessible from $\cU^+$ and $\cU^-$ so that a connected set which intersects the boundary of $\cA$
will necessarily intersect both boundaries. Recall that a point $x \in \partial \cC$ is accessible if there exists a continuous arc
$\gamma:[0,1] \to \cA$ such that $\gamma([0,1)) \en \cA \setminus \cC$ and $\gamma(1)=x$.

Here we shall use a weaker form of accessibility which, moreover, involves the dynamics of $f$ in the annulus. We will say that a periodic point $p\in \cC$ is \emph{dynamically continuum accessible from above} (resp. \emph{dynamically continuum accessible from below}) if there exist a continuum $C_p$ such that:

\begin{itemize}
\item $p\in C_p$
\item $C_p\setminus \cC$ is non empty and contained in $\cU^+$ (resp. $\cU^-$).
\item $C_p$ is inessential in $\A$.
\item $f^{n_p}(C_p) \en C_p$ for $n_p$ the period of $p$.
\end{itemize}


Using the prime-end theory and the result stated in the paragraph \ref{ss.continua}, one can show the following result.

\begin{prop}\label{prop-primeend}
Let $p$ and $q$ in $\partial \cC$ be periodic points of $f$ which are both dynamically continuum accessible from above (resp. from below). Then, for any lift of $f$ to $\R^2$ both $p$ and $q$ have the same rotation number.
\end{prop}

\begin{proof}

Assume by contradiction that $p$ and $q$ have different rotation numbers for some lift. Considering an iterate $f^j$ and a suitable lift \ro{$G$} of $f^j$ to $\R^2$ we can assume that $\ro{G}(\tilde p) = \tilde p$ and $\ro{G}(\tilde q)= \tilde q + k$ with $k\neq 0$.

\smallskip

Let $C_p$ and $C_q$ be given by the fact that $p$ and $q$ are continuum accessible from above. By definition, we have that they are disjoint and inessential. Thus, we can consider a proper arc $\gamma:[0,+\infty)\to\mathcal{U}^+ \cup \C$ so that
$\gamma(0)= z\in \C$ and $\gamma(t)\in (\C\cup C_p\cup C_q)^c$ for every $t\in(0,+\infty)$ and $\lim_{t\to +\infty}\gamma(t)=+\infty$ (see subsection \ref{ss.continua}).

\smallskip

Let $\tilde{\mathcal{U}}^+=\pi^{-1}(\mathcal{U}^+)$, $\hat{\gamma}$ a lift of $\gamma$, and $\hat{A}$ the lift
of $\pi^{-1}(\mathcal{U}^+\setminus\gamma)$ containing $\hat{\gamma}$ and $\hat{\gamma}+1$ in its boundary. Further,
consider $\hat{C}_p$ and $\hat{C}_q$ the connected components of $\pi^{-1}(C_p)$ and $\pi^{-1}(C_q)$ intersecting $\hat{A}$ respectively, $\hat{K}_p=\hat{C}_p\cap\tilde{\mathcal{U}}^+$ and $\hat{K}_q=\hat{C}_q\cap\tilde{\mathcal{U}}^+$. We have that
$\ro{G}(\hat{K}_p)\subset \hat{K}_p$ and $\ro{G}(\hat{K}_q)\subset \hat{K}_q+k$.

\smallskip

As $\hat{C}_p$ and $\hat{A}$ are in the situation of Proposition \ref{prop-levacotado}, we have that $\hat{C}_p$ intersects only finitely many of the sets $\hat{A}+j,\ j\in\Z$, and the same holds for $\hat{C}_q$.

\smallskip

Consider the map $\ro{H}:\mathbb{H}^2\to\mathbb{H}^2$ induced by $\ro{G}$ and the prime-end compactification of $\mathcal{U}^+$
as stated in sub-section \ref{ss.primeend}, and let $\ro{\theta}:\tilde{\mathcal{U}}^+\to \mathbb{H}^2$ be the induced conjugacy between $\ro{G}|_{\tilde{\mathcal{U}}^+}$ and
$\ro{\hat{H}}|_{\mathbb{H}^2}$. As $\gamma$ lands at an accesible point $z$, we have that
$\eta=\ro{\theta}(\hat{\gamma}\setminus \hat{\gamma}(0))$ can be extended continuously in $t=0$, so that $\eta(0)\in\partial\mathbb{H}^2$ with respect to the usual topology of $\R^2$ (see \ref{ss.primeend}).

\smallskip

Then, we have that the sets $\overline{K}_p=\textrm{cl}[\ro{\theta}(\hat{K}_p)]$ and $\overline{K}_q=\textrm{cl}[\ro{\theta}(\hat{K}_q)]$ are contained in a region of $\textrm{cl}[\mathbb{H}^2]$ between the extended curves $\eta-j_0$ and $\eta+j_1$ for some $j_0,j_1\in\N$. Furthermore, if $\ro{\hat{H}}$ is the continuos extension of $\ro{H}$ to $\textrm{cl}[\HH^2]$, we can assume without loss of generality, that $\ro{\hat{H}}^n(\overline{K}_p)\subset\overline{K}_p$ and $\ro{\hat{H}}^n(\overline{K}_q)\subset \overline{K}_q+nk$ for all $n\in\N$. Let $\ro{h}$ be the restriction of
$\ro{\hat{H}}$ to $\partial \mathbb{H}^2$, which is known to lift an orientation preserving circle homeomorphism as stated
in \ref{ss.primeend}.

\smallskip

Thus we obtain two compact sets $L_p=\overline{K}_p\cap\partial\mathbb{H}^2$ and $L_q=\overline{K}_q\cap\partial\mathbb{H}^2$, so that $\ro{h}^n(L_p)\subset L_p$ and $\ro{g}^n(L_q)\subset L_q+kn$ for all
$n\in\N$, which is impossible, as $\ro{h}$ lifts an orientation preserving circle homeomorphism.

\end{proof}

We are now ready to complete the proof of Theorem A by showing the following proposition.

\begin{prop}\label{buenosconexos}
There exists two periodic points $p_0,p_1$ in $\partial \C$ with different rotation numbers so that $C_{p_0}$ and $C_{p_1}$ satisfy  \eqref{eq:continuossalen}.
\end{prop}

\begin{proof} Pick four rational points $r_0,r_1,r_2,r_3 \in\rho_{\C}(F)$ with different denominators in their irreducible form (in particular, different from each other). \ve{Using Theorem \ref{t.realizacion} we know that all four are \az{realised} by periodic points $p_i$ in $\partial\C$}, and using Proposition \ref{prop-primeend} we know that at least two of them, say $p_0$ and $p_1$ are not dynamically continuum accessible.

Consider the compact connected sets $C_{p_0}$ and $C_{p_1}$ given by Theorem \ref{t.continuum}, since $p_0$ and $p_1$ are not continuum accessible, it follows directly that equation \eqref{eq:continuossalen} is verified as desired.






\end{proof}



\subsection{Proof of Theorem \ref{t.continuum}}

Let $\partial^+\mathcal{A}$ and $\partial^-\mathcal{A}$ be the two boundaries of $\mathcal{A}.$ Let $\F^+_0$ be a foliation by essential simple closed curves in the upper connected component of $\overline{\mathcal{A}\backslash f(\mathcal{A})}$ such that they coincide in the boundary with $\partial^+\mathcal{A}$ and $f(\partial^+\mathcal{A})$ and let

$$\F^+=\bigcup_{n\ge 0}f^n(\F^+_0).$$

In a symmetric way we define $\F^-.$ Notice that any annulus $\mathcal{A}_1$ \ro{whose boundary is given by a curve of $\F^+$ and curve of $\F^-$} satisfies $f(\mathcal{A}_1)\subset \textrm{int} (\mathcal{A}_1).$

\smallskip

From now on we fix a periodic point $p\in \partial \C$ as in Theorem \ref{t.continuum}. Replacing $f$ by an iterate and choosing an appropriate lift $F$
we may assume that $p$ is fixed and rotates zero. Let $q\in\partial\C$ be another periodic point with different rotational speed. We may assume without loss of generality
that $q$ is fixed and rotates one.

\begin{lemma}\label{l.arco}

There exist $\eta>0$, an annulus $\mathcal{A}_1$ bounded by leaf of $\F^+$ and a leaf of $\F^-$ and an arc $I_q\subset \mathcal{A}_1$ containing $q$ and joining both boundaries of $\mathcal{A}_1$ such that, if $g$ is $\eta$-$C^0$-close to $f$ and $G$ is the lift $\eta$-close to $F$ we have that  $G(\hat{I}_q)$ is to the right of $\hat{I}_q$ and $G^2(\hat{I}_q)$ is to the right of $\hat{I}_q+1$ in $\pi^{-1}(\cA_1)$ where $\hat I_{q}$ denotes a connected component of the lift of $I_q$.

\end{lemma}

\begin{proof}

Let $\epsilon>0$ be such that $B(p,\epsilon)\cap B(q,\epsilon)=\emptyset.$ Let $\delta$ be small enough such that $f(B(q,\delta))$ and $f^2(B(q,\delta))$ are contained in $B(q,\epsilon/2)$ (recall $f(q)=q$). Denote $B=B(q,\delta)$.

One can choose unique leaves $\F^+_\delta$ and $\F^-_\delta$  of $\F^+$ and $\F^-$ which \ro{intersect} $\partial B$, and do not intersect $B$.

We may assume (reducing $\delta$ if necessary) that both leaves also intersect $B(p,\epsilon)$ and consider the annulus $\mathcal{A}_1$ determined by $\F^+_\delta$ and $\F^-_\delta$. Denote by $K$ the connected component of $B(q,\epsilon)\cap\mathcal{A}_1$ that contains $B.$ Notice that $K$ is inessential in $\mathcal{A}_1$ since it is disjoint from $B(p,\epsilon)$ (and there is an arc in $B(p,\epsilon)$ joining the two boundaries of $\mathcal{A}_1).$

\smallskip

Let $\eta>0$ be small enough such that if $g$ is $\eta$-$C^0$ close to $f$ in $\mathcal{A}$ then:
\begin{itemize}
\item $g(\mathcal{A}_1)\subset \textrm{int}(\mathcal{A}_1).$
\item $g(B(q,\delta))$ and $g^2(B(q,\delta))$ are contained in $B(q,\epsilon)$
\item $g(B(q,\delta))\cap B(q,\delta)\neq\emptyset$ and $g^2(B(q,\delta))\cap B(q,\delta)\neq\emptyset.$
\end{itemize}
Let $I_q$ be an arc inside $B(q,\delta)$ joining the two boundaries of $\mathcal{A}_1.$ Notice that $g(I_q)$ and $g^2(I_q)$ are both contained in $K.$
Now, fix a lift $\hat{q}$ of $q$ and $\hat{I}_q$ a lift of $I_q$ containing $\hat{q}$ and let $\hat{K}$ be the connected component of $\pi^{-1}(K)$ that contains $\hat{I}_q.$ Let $G$ be the lift of $g$ which is $\eta$-close to the lift $F$ of $f.$ Since $F(\hat{q})=\hat{q}+1$ we have that $F(\hat{I}_q)\subset \hat{K}+1$ and $F^2(\hat{I}_q)\subset \hat{K}+2,$ and the same holds for $G$ which completes the proof.

\end{proof}

\begin{obs}\label{rem-neighborhood}
By continuity, one can assume without loss of generality that there exists a \az{neighbourhood} $N(I_q)$ of $I_q$ \ro{in the annulus $\mathcal{A}_1$} such that if $N(\hat I_q)$ denotes a connected component of the lift, then $G(N(\hat{I}_q))$ is to the right of $\hat{I}_q$ and $G^2(N(\hat{I}_q))$ is to the right of $\hat{I}_q+1$.
\end{obs}

From now on we fix the annulus $\mathcal{A}_1$ given by the previous lemma. The idea will be to approach $f$ by homeomorphisms presenting a stable manifold of $p$ escaping $\cA_1$ and not intersecting $I_q$ so that we will control its convergence in the limit. 

\begin{lemma}\label{l.approx}

There exists a sequence of homemorphisms $f_n$ converging to $f$ in the $C^0$ topology such that:

\begin{enumerate}

\item $p$  is a hyperbolic fixed point of $f_n.$
\item $W^s(p,f_n)$ intersects the boundary of $\mathcal{A}_1.$

\end{enumerate}

\end{lemma}

\begin{proof}

Let $\epsilon_n$ be a positive sequence converging to zero. We may assume that $B(p,2\epsilon_n)\subset \mathcal{A}_1$
for every $n.$ Let $\F^+_{\epsilon_n}$ and  $\F^-_{\epsilon_n}$ be the unique leaves of the foliations $\F^+$ and $\F^-$ which intersect $\partial B(p,\epsilon_n)$,
and do not intersect $B(p,\epsilon_n)$. Let $\mathcal{A}_{\epsilon_n}$ be the annulus determined by those leaves. Now consider $g_n$ such
that $g_n=f$ outside $B(p,\epsilon_n/2)$  and $p$ is a hyperbolic fixed point of $g_n.$ The $C^0$ distance between $g_n$ and $f$ is bounded by \ro{$\epsilon_n.$}

Fix a fundamental domain $D^s$ of $W^s(p,g_n)$ inside $B(p,\epsilon_n/2)$ and join an interior point $z$ of $D^s$ with a point $y$ in $\F^+_{\epsilon_n} \cap\partial B(p,\epsilon_n)$ by a poligonal arc inside $B(p,\epsilon_n)$, see
figure \ref{perturbation}. Let $U$ be a neighborhood of this arc which  does not intersect the forward iterates $g^m_n(D^s),m\ge 1$ and such that $\overline{U}$ is contained in the interior of $g_n^{-1}(\mathcal{A}_{\epsilon_n})$, which is equal to $f^{-1}(\mathcal{A}_{\epsilon_n}).$ We may assume that $U\subset B(p,2\epsilon_n)$ as well. See figure \ref{perturbation}.

Consider $\varphi:\mathcal{A}\to \mathcal{A}$ such that $\varphi\equiv \mathrm{id}$ outside $U$ and $\varphi(y)=z.$ The $C^0$ distance between $\varphi$ and the identity is bounded by $2\epsilon_n.$ Let $f_n=g_n\circ\varphi.$ We have that $y\in W^s(p,f_n)$ and $f_n^{-1}(y)$ belongs to the boundary of $f^{-1}(\mathcal{A}_{\epsilon_n}).$ Since $g_n=f_n$ outside $U$ and $g_n=f$ outside $B(p,\epsilon_n/2)$, iterating backwards we eventually have that $W^s(p,f_n)$ intersects the boundary of $\mathcal{A}_1.$

Finally, it is clear that the $C^0$ distance from $f_n$  to $f$ goes to zero with $\epsilon_n$ as desired.
\end{proof}

\tiny

\begin{figure}[ht]\begin{center}
 \psfrag{p}{$p$}
 \psfrag{Bp}{$B(p,\epsilon_n)$}
 \psfrag{U}{$U$}
 \psfrag{F+}{$\F^+_{\epsilon_n}$}
 \psfrag{F-}{$\F^-_{\epsilon_n}$}
 \psfrag{fF+}{$f^{-1}(\F^+_{\epsilon_n})$}
 \psfrag{fF-}{$f^{-1}(\F^-_{\epsilon_n})$}
 \psfrag{Ws}{$W^s(p,g_n)$}
 \centerline{\includegraphics[height=7cm]{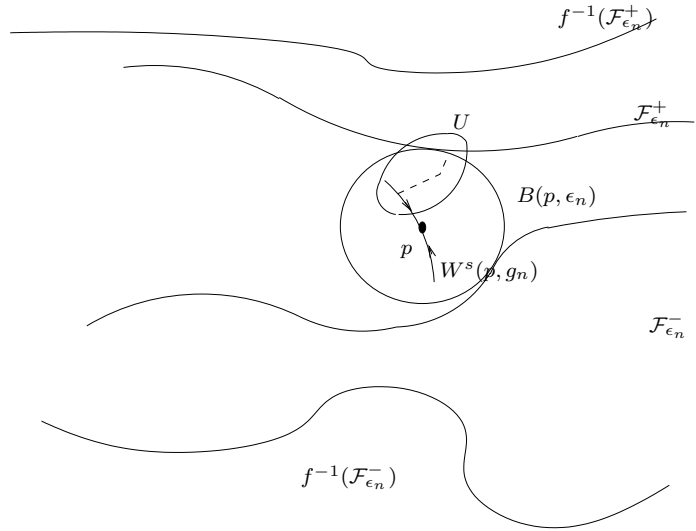}}
\caption{Construction of small perturbations having fixed hyperbolic saddles with stable manifolds accumulating at $-\infty$ or $+\infty$.}\label{perturbation}
\end{center}\end{figure}

\normalsize

Denote by $W^s_1(p,f_n)$ the connected component of $W^s(p,f_n)\cap\mathcal{A}_1$ that contains $p.$

\begin{obs}\label{rem-invariantWs}
The set $W^s_1(p,f_n)$ verifies that $f_n(W^s_1(p,f_n))\en W^s_1(p,f_n)$. Indeed, $f_n(\cA_1) \en \cA_1$ and $W^s(p,f_n)$ is also $f_n$-invariant.
\end{obs}

We now use Lemma \ref{l.arco} to control the diameter of $W^s_1(p,f_n)$ in order to be able to consider a limit continuum through $p$ which will be fordward invariant by $f$.

\begin{lemma}\label{l.bounded} Let $\mathcal{A}_1$ and $I_q$ be as in Lemma \ref{l.arco}. Then, there is a \az{neighbourhood} $N(I_q)$ of $I_q$ such that
$W^s_1(p,f_n)\cap  N(I_q)=\emptyset$ for every large enough $n.$
\end{lemma}

\begin{proof}

In the lift  $\tilde{\mathcal{A}_1}$ of $\cA_1$, we choose $\hat{p}$ in the fundamental domain $D$ determined  by a connected component $\hat{I}_q$ of the lift of $I_q$ and $\hat{I}_q-1$.

Consider a lift $\hat{W}^s_1(\hat{p},f_n)$ of $W^s_1(p,f_n)$ through $\hat{p}.$ Let $W$ be the connected component of $\hat{W}^s_1(\hat{p},f_n)\cap D$ that contains $\hat p$ and let $F_n$ be a lift of $f_n$ close to the lift $F$ of $f.$ Notice that $F_n(W)\subset W.$ We may assume that $f_n$ is $\eta$-close to $f$ where $\eta$ is as in Lemma \ref{l.arco}.

Choose $N(I_q)$ as in remark \ref{rem-neighborhood}. Assume that $\hat{W}^s_1(\hat{p},f_n)\cap N(\hat{I}_q) \neq\emptyset$. Then $W\cap N(\hat{I}_q)\neq\emptyset.$ But then $F_n(W)\subset W\subset D$. Since $F_n(N(\hat{I}_q))$ is to the right of $\hat{I}_q$ Lemma \ref{l.arco} implies that $F_n(W)$ is not contained in $D,$ a contradiction. If $\hat{W}^s_1(\hat{p},f_n)\cap (\hat{I}_q-1)\neq\emptyset $ we arrive to a contradiction as well, since then $F_n^2(W)$ is contained in $W$ and contains a point in $F_n^2(\hat{I}_q-1)$ which is to the right of $\hat{I}_q$ and so it must intersect $\hat{I}_q.$
\end{proof}

\textbf{End of proof of Theorem \ref{t.continuum}:} we say that a set $S\subset \R^2$ has \emph{bounded horizontal diameter}
if its projection to the first coordinate is bounded. In this case, let us denote $\textrm{diam}_H(S)=\textrm{diam}(\pi_1(S))$.
We consider the lift $\hat{\mathcal{A}}$ of $\mathcal{A}$. Let $\mathcal{A}_1$ be as  in Lemma \ref{l.arco} and let $\hat{\mathcal{A}}_1$ be its lift inside $\hat{\mathcal{A}}.$

\smallskip

In this context, we have that the fundamental domain in $\hat{\mathcal{A}}_1$ determined by $\hat{I}_q-1$ and $\hat{I}_q$ has bounded horizontal diameter, say by $a> 0$. This implies, by Lemma \ref{l.bounded} that $\textrm{diam}_H(\hat{W}^s_1(p_n,f_n))$ is also bounded by $a.$

\smallskip

Let $m$ be the first positive integer such that $f^m(\mathcal{A})\subset \mathcal{A}_1.$ Notice that $f^m_n(\mathcal{A})\subset\mathcal{A}_1$ by construction. Let $F$ be the lift of $f$ and $F_n$ the lift of $f_n.$ Then $F^{-m}_n(\hat{W}^s_1(p,f_n))$ has bounded diameter in $\R^2.$ Let $\hat C_n=F^{-m}_n(\hat{W}^s_1(p,f_n)).$ We  have that:
\begin{enumerate}
\item $\hat C_n$ is a continuum containing $\hat{p}.$
\item $\hat C_n$ is forward invariant by $F_n$ (c.f. remark \ref{rem-invariantWs}).
\item $\hat C_n$ intersects the boundary of $\hat{\mathcal{A}}.$
\item $\hat C_n$  has uniformly bounded diameter.
\item $F_n\rightrightarrows F.$
\end{enumerate}

Then, by taking the Hausdorff limit $\hat{C}_p$ of $(\hat C_n)_{n\in\N}$ we have a continuum which is forward invariant under $F$ and contains $\hat{p}$. Moreover, it intersects $\partial \hat{\mathcal{A}}$, and its projection into $\A$ must be inessential, since otherwise it would intersect $I_q$ which is not possible. Taking $C_p=\pi(\hat{C}_p)$ we are done.

\lqqd

\ve{
\subsection{General Statement for Theorem A}\label{ss.teoC}

In this section we comment on the proof of Theorem A to see that weaker hypothesis are enough to obtain the existence
of rotational horseshoes. We state a general version of the result from which Theorem C can be obtained.

Consider an $f\in\textrm{Homeo}_+(\A)$ such that $f(\mathcal{A})\subset\textrm{int}(\mathcal{A})$ for
some compact and essential annulus $\mathcal{A}$. In this situation an attractor
$K_{\mathcal{A}}=\bigcap_{n\in\N}f^n(\mathcal{A})$ exists and is an essential annular continuum.

The proof of Theorem \ref{t.continuum} extends to the following with the same proof.

\begin{thm}\label{t.continuum2}

Assume we are in the situation above, and $p,q\in K_{\mathcal{A}}$ are two periodic points of $f$ so that

\begin{itemize}

\item $p$ and $q$ have different rotation vectors for any lift $F$ of $f$,

\item $p$ and $q$ are both an accumulation point of $\bigcup_{n\in\N}f^n(\partial^+ \mathcal{A})$ and of $\bigcup_{n\in\N}f^n(\partial^- \mathcal{A})$.

\end{itemize}

Then there exists an inessential continuum $C_p$ containing $p$ such that $f^k(C_p)\subset C_p$ where
$k$ is the period of $p$ and $C_p\cap\partial A\neq\emptyset$. A similar statement holds for $q$.

\end{thm}

The following is an easy application of Zorn's Lemma and Theorem \ref{teo-circindec}.

\begin{lemma}\label{lemmazorn}

Let $f\in\textrm{Homeo}_+(\A)$ and a closed essential annulus $\mathcal{A}$ such that $f(\mathcal{A})\subset\textrm{int}(\mathcal{A})$.
Further, assume that there are at least four periodic points $p_1,p_2,p_3,p_4$ in $\mathcal{A}$ having pairwise different rotation vectors for any lift $F$ of $f$,
and that $\bigcup_{n\in\N}f^n(\partial^i\mathcal{A})$ accumulates in $p_1,p_2,p_3$ and $p_4$ for $i=+,-$.
Then, there exists an invariant a non-compactly generated circloid $\C\subset K_{\mathcal{A}}$ so that $p_i\in\C$ for $i=1,2,3,4$.

\end{lemma}

With these two results, following exactly the proof of Theorem A, we obtain a more general result.

\begin{thm}\label{t.generalA}

Let $f\in\textrm{Homeo}_+(\A)$ and a closed essential annulus $\mathcal{A}$ such that $f(\mathcal{A})\subset\textrm{int}(\mathcal{A})$.
Further, assume that there are at least four periodic points
$p_1,p_2,p_3,p_4$ in $\mathcal{A}$ having pairwise different rotation vectors for any lift $F$ of $f$, and that
$\bigcup_{n\in\N}f^n(\partial^i\mathcal{A})$ accumulates in $p_1,p_2,p_3$ and $p_4$ for $i=+,-$.
Then, there exists $n_0\in\N$ and a $C^0$-neighborhood $\mathcal{N}$ of $f$ in $\textrm{Homeo}_+(\A)$ such that for
every element $g\in\mathcal{N}$ the power $g^{n_0}$ has a rotational horseshoe $\Lambda_g\subset\mathcal{A}$. In particular, every
element in $\mathcal{N}$ has topological entropy larger than $\frac{\mathrm{log}(2)}{n_0}$.

\end{thm}

We finish this paragraph by showing that Theorem C can be derived from this las theorem as well. In the hypothesis
of Theorem C we have that for the closed annulus $\mathcal{A}$ the attractor $K_{\mathcal{A}}$ must have empty-interior
as the map is dissipative. Furthermore, as the Birkhoff attractor $\C$ is by definition the unique circloid contained
in $K_{\mathcal{A}}$ and has non-empty interior, it must be that $\C=\overline{U^+}\cap \overline{U^-}$ where $U^+,U^-$ are the connected
components of $\A\setminus K_{\mathcal{A}}$. This implies that both sets $\bigcup_{n\in\N}f^n(\partial^+ \mathcal{A})$
and $\bigcup_{n\in\N}f^n(\partial^+ \mathcal{A})$ accumulates on every point of $\C$. As the rotation
set on $\C$ is not trivial, the realisation results \cite{kororealizacion,BG} imply that we have infinitely many periodic
points in $\C$ realising every rational number in $\rho_{\C}(F)$, for any lift $F$ of $f$. Hence the last theorem can be applied, so we obtain Theorem C.

}

\section{Entropy versus rotation set for circloids.}\label{chapterperttwist}

Let us recall the basic definitions. We considered $\A=\mathbb{S}^1 \times \R$ where $\mathbb{S}^1 = \R/_{\Z}$,
and the usual covering $\pi:\R^2\to\A$ given by $\pi(x,y)= (x\ (\text{mod } \Z),y)$.

\smallskip

Consider the \emph{integrable twist map} $\tau:\A\to\A $ given by the lift:
$$T(x,y)=(x+y ,y) \,$$

If we denote by $\mathcal{F}=\{C_y\}_{y\in\R}$ the foliation of $\A$ by essential circles given by $C_y=\pi(\R \times \{y\})$, we have that $\tau|_{C_y}$ is a rotation of angle \ro{$2\pi y$}.

\begin{obs}\label{rem-twist}
A simple computation gives that $\lim_n\frac{1}{n} \log \|D\tau^n\| = 0$. This can be combined with
Proposition \ref{p.continuityentropy} to get that given $\eps>0$ there is a $C^1$-neighborhood $\mathcal{N}_\eps$ of
$\tau$ such that $h_{top}(f) < \eps$ for all $f \in \mathcal{N}_\eps$ .
\end{obs}






We will prove the following theorem.

\smallskip

\begin{thm}\label{maincircloidlowent}
For every $C^1$-neighborhood $\mathcal{N}$ of $\tau$ there exists $f \in \mathcal{N}$
so that $f$ has a global attractor given by an essential circloid $\mathcal{C}$
with $\rho_{\mathcal{C}}(F)\supset [0,1]$ for some lift $F$ of $f$.
\end{thm}

\smallskip

Combining this with  remark \ref{rem-twist}, we show that there are circloids with rotation sets containing $[0,1]$ whose entropy approaches zero as much as desired,
therefore proving Theorem B.  Notice that the twist condition is $C^1$-open so that we can assume also that the obtained diffeomorphism verifies the twist condition.
To obtain a dissipation hypothesis (as required in Theorem D) one has to perform a slightly different perturbation which is explained at the end of this section.

\smallskip

We fix $\mathcal{N}$ and construct $f\in\mathcal{N}$ by means of a sequence of $C^1$ perturbations of $\tau$. We remark that all the perturbations are just $C^1$ small,
but the map itself can be considered to be smooth (see remark \ref{remark:smooth}).

\subsection{First perturbation}
We first fix some notation. For $y <y'$ we denote by $[C_y,C_{y'}]$ to the compact region between these two circles, and by $(C_y,C_{y'})$ its interior.

\smallskip

As usual, given a map $f:M\to M$ and a point $x\in M$, we define for $\varepsilon>0$ the local stable set of $x$ by
$W^s_{\varepsilon}(x,f)=\{y\in M\ |\ d(f^n(y),f^n(x))<\varepsilon\mbox{ for all }n\in\N\}$, and define the stable set of $x$
as $W^s(x,f)=\{y\in M\ |\ \lim_n d(f^n(x), f^n(y))=0\}$. The local unstable and unstable sets are defined by considering $f^{-1}$
instead of $f$, i.e. $W^u(x,f)=W^{s}(x,f^{-1})$. When $x$ is a hyperbolic periodic point, the local stable set is a sub-manifold tangent to the stable subspace at $x$,
and it holds $W^s(x,f)=\bigcup_{n\in N}f^{-kn}(W^s_{\varepsilon}(x,f))$ where $k$ is the period of $x$ (see \cite[Section 6]{KatokHasselblatt}).
A similar result holds for unstable manifolds.

\medskip

The first perturbation will be $f_1\in\mathcal{N}$ so that (see Figure \ref{efe1}):

\begin{enumerate}

\item $f_1$ is conservative restricted to the annulus $[C_0,C_1]$.

\item $f_1(C_r)=C_r$ for $r\in\{0,1\}.$

\item $f_1$ has a saddle $x_{0}\in C_{0}$ and a saddle-node $p_{0}\in C_{0}$, so that
 $W^u(x_{0},f_1)=C_{0}\setminus p_{0}$, which implies that $W^s(p_{0},f_1)\supseteq C_{0}\setminus\{x_{0}\}$.

\item $f_1$ has a saddle $x_{1}\in C_{1}$ and a saddle-node $p_{1}\in C_{1}$, so that
 $W^u(x_{1},f_1)=C_{1}\setminus p_{1}$, which implies as before that $W^s(p_{1},f_1)\supseteq C_{1}\setminus\{x_{1}\}$.


\item There is a forward invariant arc $I_{0}^s\subset W^{s}(p_{0},f_1)\cap (-\infty,C_{0}]$ with one endpoint at $p_{0}$, and a small backward invariant compact arc $I_{0}^u\subset W^u(p_{0},f_1)\cap[C_0,C_1]$ with one endpoint in $p_{0}$.

\item There is a forward invariant arc $I_{1}^s\subset W^{s}(p_{1},f_1)\cap [C_{1},+\infty)$ with one endpoint at $p_{1}$, and a small backward invariant compact arc $I_{1}^u\subset W^u(p_{1},f_1)\cap[C_0,C_1]$ with one endpoint in $p_{1}$.

\item $[C_{0},C_{1}]$ is a global attractor for $f_1$.

\item For every $n\in\N$, $f_1$ has finitely many points of period $n$.

\end{enumerate}

This can be done by $C^1$-small smooth perturbations around the circles $C_{0}$ and $C_{1}$ and the Franks' lemma \cite{FranksLemma} (see \cite[Proposition 7.4]{BDP}
for the conservative version) for suitable perturbations of the derivative in the conservative setting.
To obtain (7), one can just take a dissipative perturbation supported in $(C_0,C_1)^c$. Item (8) can be achieved
by means of usual arguments in generic dynamics: a simple Baire argument allows to find a smooth diffeomorphism nearby for which all periodic points in the
interior of the annulus have no eigenvalues equal to $\pm 1$, and this implies that the set of those having period $n$ is finite for all $n\in\N$.
This first perturbation is depicted in Figure \ref{efe1}.
\small

\begin{figure}[ht]\begin{center}
 \psfrag{C0}{$C_0$}\psfrag{x0}{$x_{0}$}\psfrag{p0}{$p_0$}\psfrag{xdelta}{$x_{1/4}$}\psfrag{x1menosdelta}{$x_{3/4}$}
 \psfrag{C1}{$C_{1}$}\psfrag{x1}{$x_1$}\psfrag{p1}{$p_1$}
 \psfrag{inf}{$+\infty$}

\centerline{\includegraphics[height=9cm]{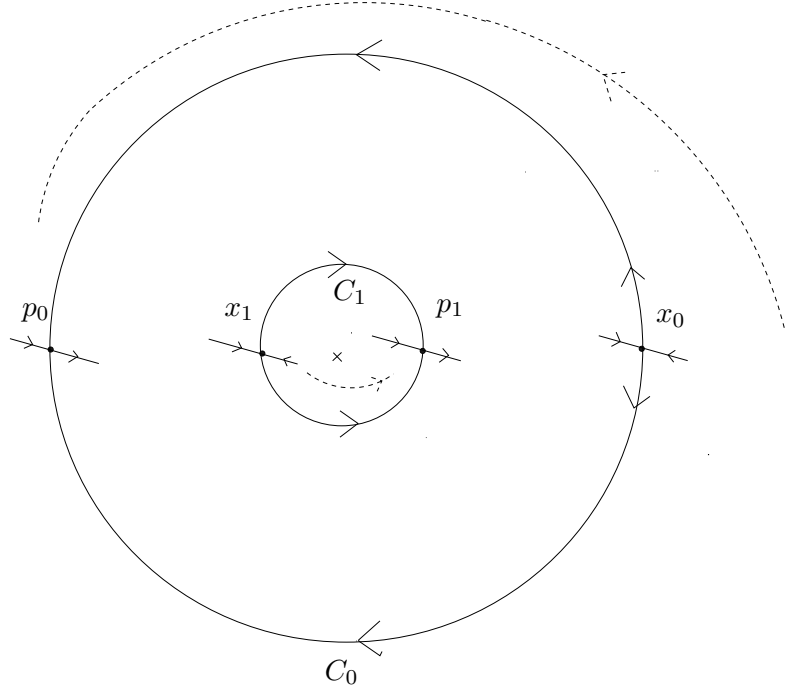}}
\caption{The map $f_1$.}\label{efe1}
\end{center}\end{figure}

\normalsize

\subsection{Second perturbation}

For the second perturbation, we make use of  Theorem \ref{connecting2}. We construct $f_2\in\mathcal{N}$ so that:

\begin{enumerate}

\item $f_2$ is conservative in $[C_0,C_1]$,

\item $f_2(x)=f_1(x)$ outside $(C_{r_1},C_{r_2})$ for some values $0<r_1 <r_2 <1$,

\item there is a transverse intersection between the connected component of $I^u_0 \cap [C_0,C_{r_1}]$ containing $p_0$ and $W^s(x_1,f_2)$ and a transverse intersection between the connected component of $I^u_1 \cap [C_{r_2},C_{1}]$ containing $p_1$ and $W^s(x_0,f_2)$.



\end{enumerate}

\begin{obs}\label{rem-instabilityregion}
The diffeomorphism $f_2$ restricted to $[C_0,C_1]$ is a conservative annulus diffeomorphism which deviates the vertical and the whole annulus is an instability region. In particular, the rotation set in this instability region is $[0,1]$ and the entropy can be chosen to be as small as desired.
\end{obs}

\small

\begin{figure}[ht]\begin{center}
 \psfrag{C0}{$C_0$}\psfrag{x0}{$x_{0}$}\psfrag{p0}{$p_0$}\psfrag{xdelta}{$x_{1/4}$}\psfrag{x1menosdelta}{$x_{3/4}$}
 \psfrag{C1}{$C_{1}$}\psfrag{x1}{$x_1$}\psfrag{p1}{$p_1$}
 \psfrag{inf}{$+\infty$}

\centerline{\includegraphics[height=9cm]{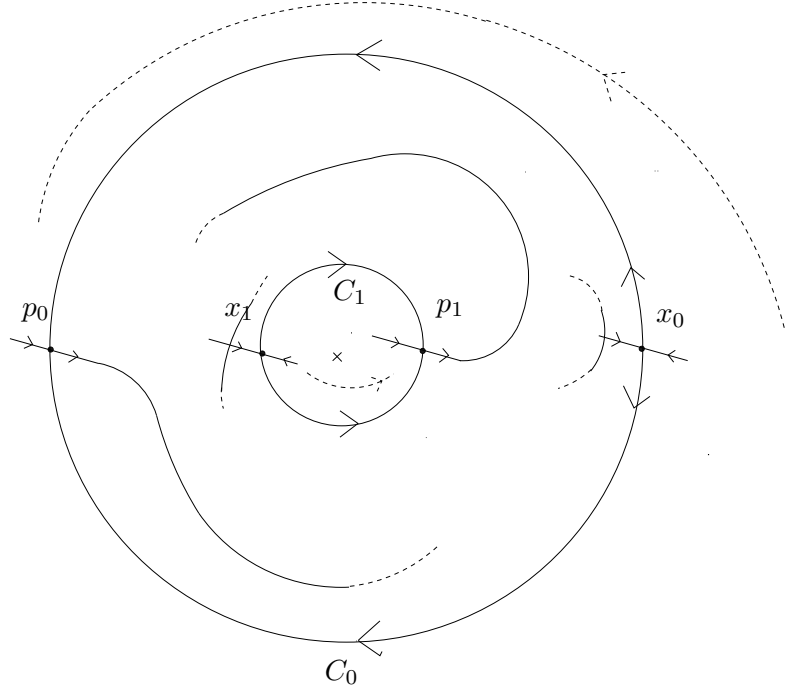}}
\caption{The map $f_2$.}\label{efe2}
\end{center}\end{figure}

\normalsize

In order to produce $f_2$ we just have to choose a perturbation of $f_1$ in $\mathcal{N}$ which is conservative in $[C_0,C_1]$, supported outside a \az{neighbourhood} of $C_0$ and $C_1$ in $[C_0,C_1]$ and connects the forward orbit of a small arc in $I^u_0$ (inside the \az{neighbourhood} where the perturbation is made) with the stable manifold of $x_1$ and symmetrically connects the forward orbit of $I^u_1$ with the stable manifold of $x_0$ . See figure \ref{efe2}.

\smallskip

This will be achieved by means of Theorem \ref{connecting2}. But first we need to show an abstract lemma to put ourselves in the hypothesis of the theorem.

\begin{lema}\label{connectingK}
Assume $h: [C_0,C_1] \to [C_0,C_1]$ is an area preserving diffeomorphism and $D$ is a connected open subset whose closure is contained in $(C_0,C_1)$. Let $z,w$ be points in $(C_0,C_1)$ such that there are integers $n_z>0$ and $n_w>0$ so that $h^{n_z}(z)$ and $h^{-n_w}(w)$ are contained in $D$. Then $z \dashv_{\textrm{cl}[D]} w$.
\end{lema}

\begin{proof} Notice that it is enough to show that for every pair of points $p$ and $q$ in $D$ and $\eps>0$ one can construct a pseudo-orbit with jumps in $D$ going from $p$ to $q$ since one can go without jumps from the interior of $D$ to the points $z$ and $w$.

We fix $p$ in $D$, and consider for every $\varepsilon>0$ the set $P_{\varepsilon}$ of those points $q\in D$ so that
there exists an $\varepsilon$-pseudo-orbit $(z_k)_{k=0}^{n}$ with $z_0=p,\ z_n=q$ and $h(z_k),\ z_{k+1}\in D$ whenever
$h(z_k)\neq z_{k+1}$. It is enough to prove that $P_{\varepsilon}$ is a non-empty open and closed set in $D$.

\smallskip

For $q\in P_{\varepsilon}$ we can consider an $\varepsilon$-pseudo-orbit $(z_k)_{k=0}^n$ as before. Then,
there exists $\varepsilon'$ such that $d(h(z_{n-1}),q)<\varepsilon'<\varepsilon$. Pick a \az{neighbourhood} $V$ of $q$ in $D$,
so that $V\subset B(q,\varepsilon-\varepsilon')$ and take $z\in V$.

\begin{itemize}

\item if $h(z_{n-1})=q$, we have that $z_0,\dots,z_{n-1},z$ is a $\varepsilon$ pseudo-orbit whose jumps are in $D$. Thus $z\in P_{\varepsilon}$, and $V\subset P_{\varepsilon}$, so $P_{\varepsilon}$.

\item If $h(z_{n-1})\neq q$, then both $h(z_{n-1})$ and $z$ are contained in $D$. Thus the pseudo-orbit $z_0,\dots,z_{n-1},z$
is a $\varepsilon$ pseudo-orbits who has it jumps in $D$. Thus, we have again $V\subset P_{\varepsilon}$.

\end{itemize}

Therefore, we can conclude that $P_{\varepsilon}$ is open. In order to check that it is also closed in $D$, we consider a sequence of points $q_n\in P_{\varepsilon}$ converging to a point $q$ in $D$. Fix $q_n$ so that $d(q_n,q)<\varepsilon$ and let $V$ be a \az{neighbourhood} of $q_n$ in $D$, contained in $B(q,\varepsilon)$.

\smallskip

Consider an $\varepsilon$-pseudo-orbit $p=z_0,\dots,z_{m}=q_n$ with jumps inside $D$. Hence, $d(h(z_{m-1}),q_n)<\varepsilon$. Poincar\'{e}'s recurrence Theorem (see \cite[Section 4.1]{KatokHasselblatt}) implies that we can consider a recurrent point $r\in V$ so that $d(h(z_{m-1}),r)<\varepsilon$. Let $h^l(r)\in V$ and define the pseudo-orbit

$$p=z_0,\dots , z_{m-1}, r,h(r),\dots,h^{l-1}(r), q\ .$$

Then, we have an $\varepsilon$-pseudo-orbit from $p$ to $q$ whose jumps are all contained in $D$.

To show that $P_\varepsilon$ is non-empty, notice that again by Poincar\'{e}'s recurrence theorem, one has that $p \in P_\varepsilon$.

\end{proof}

Now let us construct the desired perturbation of $f_1$.

\smallskip

\ro{For $f_1\in \textrm{Diff}_{\nu,per}^{\ 1}(\A)$ and the prescribed neighborhood $\mathcal{N}$ let $N=N(f_1,\mathcal{N})$
be the positive integer given by Theorem \ref{connecting2}}.
We consider first the set $D\subset(C_0,C_1)$ given by $D_0= (C_{a_0},C_{b_0})$ so that the arc $I^u_0$ and the invariant manifold $W^s(x_1,f_1)$ intersects $D_0$.  Choose $0<a_1<a_0$ and $b_0< b_1<1$ so that $D_1 = (C_{a_1},C_{b_1})$ contains $\textrm{cl}[D_0 \cup \ldots \cup f_1^{N-1}(D_0))]$.

 Choose a point $z \in I^u_0 \setminus D_1$ and $w \in W^s(x_1,f_1)  \setminus D_1$. It follows from Lemma \ref{connectingK} that one has $z \dashv_{D_0} w$. Theorem \ref{connecting2} implies that there exists $g \in \cU$ such that $g^n(z) = w$ and such that $g=f_1$ outside $D_1$. Due to the way $w$ is chosen, and since $g=f_1$ outside $D_1$, it follows that $w$ still belongs to $W^s(x_1,g)$ after perturbation\footnote{Technically one has to choose $z \neq p_0$ in the connected component of $I^u_0 \setminus D_1$ containing $p_0$ and $w\neq x_1$ in the connected component of $W^s(x_1,f_1) \setminus D_1$ containing $x_1$.} and the same holds for $I^u_0$ so we deduce that $I^u_0$ intersects $W^s(x_1,g)$. A further small perturbation makes this intersection transversal. Being transversal, the intersection will persist \ro{for sufficiently} small $C^1$-perturbations even if the involved points are moved,

 Now, we do the same argument again but reducing further $a_1$ and $b_1$ so that we can connect $I^u_1$ with the stable manifold of $x_0$ and again make the intersection transversal. We can choose the perturbation small enough so that the intersection we had already created persists thanks to transversality. This concludes the proof that $f_2 \in \mathcal{N}$ can be constructed.

\subsection{Final perturbation}
For our last move, we fix $z_{0}$ in one of the connected components of $C_{0}\setminus\{x_{0},p_{0}\}$ and $z_{1}$ in one of the connected components of
$C_{1}\setminus\{x_{1},p_{1}\}$. Consider for $k=0,1$ an open ball $B(z_k,\delta)$ so that
$B(z_k,\delta)\cap C_{k}\ro{=}I_k$ is a wandering interval, i.e., $I_k\cap \bigcup_{n\in\Z\setminus\{0\}}f_2^n(I_k)=\emptyset$.

\smallskip

We now take two $C^{\infty}$-diffeomorphisms $b_{0}$ and $b_{1}$ which are arbitrary $C^{\infty}$-close to the identity, supported in $B(z_{0},\delta)$ and $B(z_{1},\delta)$, defined as follows.

\smallskip


If we set for every $p\in\R^2$ the coordinates $\tilde{x}=\pi_1(p-z_{0})$ and $\tilde{y}=\pi_2(p-z_{0})$\footnote{Here $\pi_1$ and $\pi_2$ stay for the projections over the first and second coordinate in $\R^2$.}, the first map is given by

$$b_{0}(p)=(\tilde{x},\tilde{y}+\mu(\tilde{x},\tilde{y}))\ ,$$
where $\mu:\R^2\to [0,1]$ is some $C^{\infty}$ \emph{bump} function which \ro{is zero in $B(0,\delta)^c$ and positive
in $B(0,\delta)$. Note that $I_0\cup b_0(I_0)$ is the boundary of an open disk contained in $(C_0,C_1)$.}

\smallskip

For $b_{1}$, if we now set for every $p\in\R^2$ the coordinates $\tilde{x}=\pi_1(p-z_{1})$ and $\tilde{y}=\pi_2(p-z_{1})$, we define

$$b_{1}(p)=(\tilde{x},\tilde{y}-\mu(\tilde{x},-\tilde{y}))\ .$$
\ro{Note that $I_1\cup b_1(I_1)$ is the boundary of an open disk contained in $(C_0,C_1)$.}
\ro{Let us call} by $L_{0}$ the open disk between $I_{0}$ and $b_{0}(I_{0})$ and
$L_{1}$ the open disk in-between $I_{1}$ and $b_{1}(I_{1})$.

\smallskip

We are ready now to perform our final perturbation. We consider $f\in\mathcal{N}$ so that

$$f=b_{1}\circ b_{0}\circ f_2\ ,$$

where the following holds:

\begin{enumerate}


\item Property (3) of the second perturbation $f_2$ still holds.

\item $\lim_n f^{-n}(l)=- \infty$ for all $l\in L_{0}$,

\item $\lim_n f^{-n}(l)=+\infty $ for all $l\in L_{1}$.

\end{enumerate}

Indeed, the choice of $b_0$ and $b_1$ imply immediately the last two properties and if $b_0,\ b_1$ are small enough then the transverse
intersections required in (3) of $f_2$ still holds. Notice that $f=f_2$ in a neighborhood of $x_0$, $x_1$, $p_0$ and $p_1$. See Figure \ref{efen} for a
schematic drawing.

\small

\begin{figure}[ht]\begin{center}
 \psfrag{H}{$H$}\psfrag{x0}{$x_{0}$}\psfrag{p0}{$p_0$}
 \psfrag{Bz1delta}{$B(z_1,\delta)$}\psfrag{Bz0delta}{$B(z_0,\delta)$}\psfrag{x1}{$x_1$}\psfrag{p1}{$p_1$}
\psfrag{C0}{$C_0$}\psfrag{C1}{$C_1$}\psfrag{z1}{$z_1$}\psfrag{z0}{$z_0$}

\centerline{\includegraphics[height=9cm]{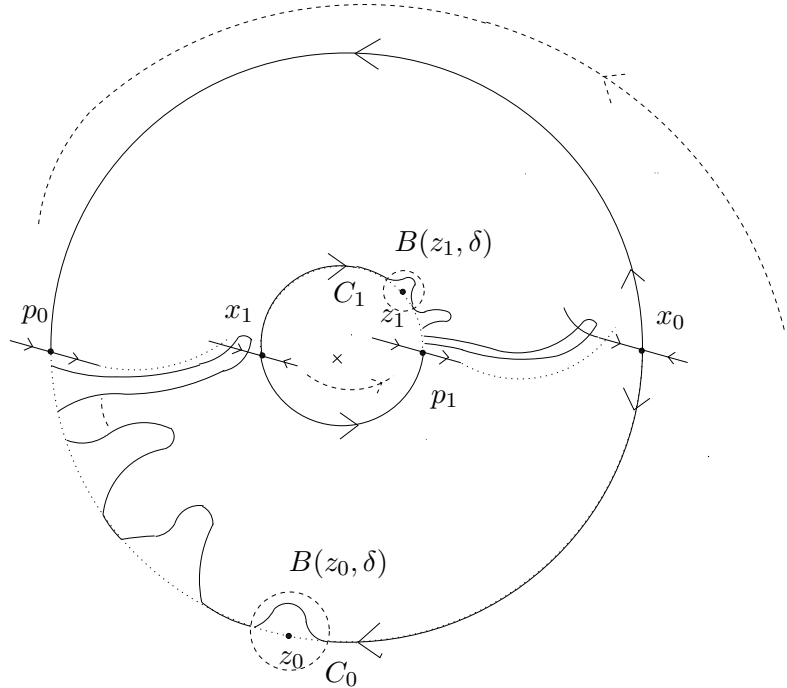}}
\caption{The final map $f$. We perform a small perturbation near $z_0,\ z_1$ so that $x_0,p_0,x_1,p_1$ belong to the same homoclinic class. The closure of $W^u(x_0,f)$ will give our desired circloid.}\label{efen}
\end{center}\end{figure}

\normalsize

\subsection{The perturbation verifies the announced properties}

We must now show that $f$ verifies our theorem \ref{maincircloidlowent}. Consider the set

$$\mathcal{B}=\textrm{cl}[W^u(x_{0},f)]\ $$

Observe that it is a closed connected set. The next lemma shows that it coincides with $\textrm{cl}[W^u(x_{1},f)]$, and by construction
$\mathcal{B}\subset [C_{0},C_{1}]$. Thus, we actually have that $\mathcal{B}$ is an essential continuum
with $\rho_{\mathcal{B}}(F)\supseteq [0,1]$ for some suitable lift $F$ of $f$. Let us call $\mathcal{U}^-$ and
$\mathcal{U^+}$ the two unbounded connected components of $\A\setminus\mathcal{B}$.

\smallskip

\begin{lemma}\label{l.homoclinicconextionx_0} The points $x_0$ and $x_1$ are homoclinically related.
\end{lemma}

\begin{proof}
This follows by applying a small variation of the $\lambda$-lemma \cite{KatokHasselblatt} in a \az{neighbourhood} of $p_0$ (resp. $p_1$).
Notice that the usual $\lambda$-lemma does not apply since $p_0$ is not hyperbolic but by looking at the local dynamics of $p_0$ and the
way we have performed the perturbation $b_0$ (far from $p_0$) one has that the new unstable manifold of $x_0$ will approach for forward
iterates the unstable manifold of $p_0$ which is connected to the stable manifold of $x_1$. The symmetric argument gives that the unstable
manifold of $x_1$ must intersect transversally the stable manifold of $x_0$.
\end{proof}

Furthermore, as we have one branch of $W^{s}(x_{0},f)$ contained in $\mathcal{U}^-$, Lemma \ref{l.homoclinicconextionx_0} shows that
$\partial \mathcal{U}^-\supseteq \mathcal{B}$. In the same way, as the saddle $x_{1}$ is homoclinically related to $x_0$ and one branch of $W^s(x_1,f)$ is contained in $\mathcal{U}^+$,
we have that $W^s(x_{0},f)$ must intersect $\mathcal{U}^+$. Therefore, arguing with the $\lambda$-lemma, we find that $\partial \mathcal{U}^+\supseteq \mathcal{B}$.  So:

$$\mathcal{B}\subset \partial \mathcal{U}^-\cap\partial\mathcal{U}^+.$$

On the other hand, since $\mathcal{U}^\pm$ is a connected component of $\A \setminus \mathcal{B}$, the set $\cU^\pm \cup \mathcal B$ is closed, and in particular, $\partial \cU^\pm \subset \mathcal{B}$, therefore, $\mathcal B = \partial \cU^- = \partial \cU^+$.

\noindent This implies, that $\mathcal{B}$ is the boundary of a circloid $\mathcal{C}$
with $\A\setminus\mathcal{C}=\mathcal{U}^{-}\cup\mathcal{U}^+$ \ve{as it is proved for instance in \cite[Corollary 3.3]{Jaeger})}. In order to obtain \ref{maincircloidlowent}, we need to prove that $\mathcal{C}$ is the global attractor of $f$. 

\medskip

For this, it is enough to show that every point $u\in \mathcal{U}^-$ has
its $\alpha$-limit in $-\infty$ and that every point $v\in\mathcal{U}^+$ has it $\alpha$-limit in $+\infty$. We work with
$\mathcal{U}^-$, the other case is similar. Recall the definition of the open disk $L_0$
associated to the wandering interval $I_0$. We have by construction that $L_0$ is bounded by the concatenation of
curves \ro{$I_0$ and $b_0(I_0)$. Denote by $\tilde{I}_0$ the maximal open interval in $I_0$.}

\smallskip

In order to show that $-\infty=\lim_n f^{-n}(u)$ for all $u\in\mathcal{U}^-$ it is enough to show following lemma.

\begin{lemma}\label{atractor}

We have that $\mathcal{U}^-=(-\infty,C_{0})\cup \bigcup_{n\in\N} f^n(L_{0}\cup \ro{\tilde{I}}_0)$.
\end{lemma}

\begin{proof}

It is easy to see that $(-\infty,C_{0})\subset\mathcal{U}^-$. Further, as $L_{0}\cup \ro{\tilde{I}}_0\subset \mathcal{U}^-$, we have
that $\mathcal{U}^-\supseteq (-\infty,C_{0})\cup \bigcup_{n\in\N} f^n(L_{0})$. We must look now for the symmetric inclusion.

\medskip

Observe that $f^n(\ro{I_0})\subset C_{0}$ for all $n\in\N$ and that $f^n(\ro{b_0(I_0)})\subset[C_{0},+\infty)\cap\mathcal{C}$ for all $n\in\N$.

\smallskip

\ro{Let $W$ be the interior of the arc in $C_0$ joining $x_0$ and $p_0$ and containing $I_0$}. Observe that the closure of the complementary connected component is contained in $\mathcal{C}$. Then it holds
$$ C_{0}\cap \mathcal{C}= C_{0}\setminus \left(\bigcup_{n\in\N} f^n(\ro{\tilde{I}_0})\right).  $$

Assume $x\in\mathcal{U}^-\cap [C_{0},+\infty)$, hence we can connect $x$ to $-\infty$ throughout simple curve $\Gamma'\subset\mathcal{U}^-$, which must contain a compact arc $\Gamma\subset [C_{0},+\infty)$ from $x$ to
certain point in $f^{n_0}(\ro{I_0})$. Thus, $\Gamma$ must be contained in a disk bounded by the concatenation of
$f^{n_0}(\ro{I_0})$ and $f^{n_0}(\ro{b_0(I_0)})$, otherwise $\Gamma$ meets $f^n(\ro{b_0(I_0)})\subset \mathcal{C}$.

\smallskip

Therefore we get that $x\in f^{n_0}(L_{0})$, and we have

$$\mathcal{U}^-=(-\infty,C_{0})\cup \bigcup_{n\in\N} f^n(L_{0}\cup\ro{\tilde{I}}_0).$$

\end{proof}

We conclude that the non-wandering set of $f$ is contained in $\mathcal{C}$, so $\mathcal{C}$ must be a global attractor for $f$, and we are done with the proof of Theorem \ref{maincircloidlowent} (and consequently of Theorem B).

\subsection{Proof of Theorem D}

We here perform some modifications to the construction developed above to obtain a proof of Theorem D. In the construction of $f_2$, it is not hard to construct another pair of
saddle periodic points inside $(C_0,C_1)$, so that they are homoclinically related and have different rotation numbers which are as close as desired to $0$ and $1$ respectively.
This can be achieved using Theorem \ref{connecting2}. See Figure \ref{ultimafigura}.

\small

\begin{figure}[ht]\begin{center}
 \psfrag{C0}{$C_0$}\psfrag{x0}{$x_{0}$}\psfrag{p0}{$p_0$}\psfrag{xdelta}{$x_{1/4}$}\psfrag{x1menosdelta}{$x_{3/4}$}
 \psfrag{C1}{$C_{1}$}\psfrag{x1}{$x_1$}\psfrag{p1}{$p_1$}
 \psfrag{inf}{$+\infty$}\psfrag{H}{$H$}\psfrag{q}{$q$}\psfrag{q1}{$q'$}

\centerline{\includegraphics[height=9cm]{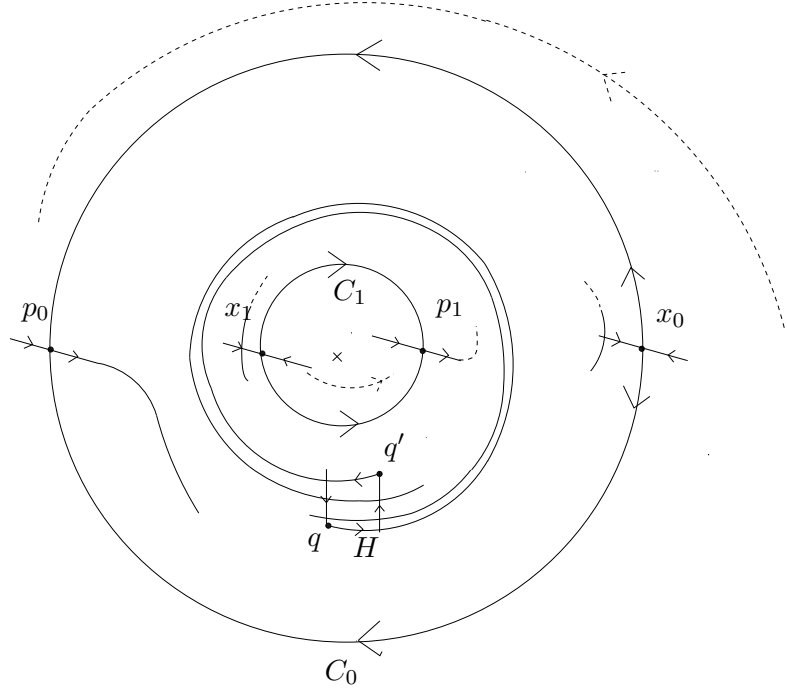}}
\caption{The map $f_2$ for the examples in Theorem D. We consider a homoclinic class $H=H(q,q')$ contained in $(C_0,C_1)$ with a rotation set arbitrary close to $[0,1]$.}\label{ultimafigura}
\end{center}\end{figure}

\normalsize

Then, \ve{for an arbitrary small} $\delta>0$ we can choose $f_2$ so that there is a homoclinic class \footnote{In our context the homoclinic class is the minimal $f$-invariant set containing the closure of the transversal heteroclinic intersections associated to the periodic points $q$ and $q'$.} \ve{$H=H(q,q')\subset(C_0,C_1)$ with $(\delta,1-\delta)\subset\rho_H(F_2)\subset [0,1]$ for a lift $F_2$ of $f_2$}. Notice that $f_2$ verifies $f_2(\mathbb{S}^1\times [-1,2]) \subset \mathbb{S}^1 \times (-1,2)$ and we can assume that the determinant of the derivative of $f_2$ is everywhere smaller than $1-\delta$ outside $\mathbb{S}^1\times [-1,2]$.

Now, instead of pushing the unstable manifolds of $x_0$ and $x_1$ we will consider smooth diffeomorphisms $h_n$ which coincide with the identity outside the region $(C_{-n}, C_{n+1})$, and having the form $h_n(x,y) = (x, \hat h_n(y))$ and $\hat h_n: \R \to \R$ is a function such that:

\begin{itemize}
\item $\hat h_n'(y) \in (1-1/n, 1+1/n)$ for every $y\in \R$ and $\hat h_n'(y) < 1-1/2n$ if $y \in [-1,2]$
\item  the $C^1$-distance between $h_n$ and the identity tends to $0$ as $n\to \infty$.
\end{itemize}

We will consider the perturbations $g_n = h_n \circ f_2$.

Since $f_2$ has the homoclinic class $H$, it follows that for large enough $n$, this class has a continuation $H_n$ which contains in its rotation set
the interval $[\delta,1-\delta]$. Moreover, for large enough $n$ there will still be a global attractor as one has $g_n ([C_{-1},C_2]) \subset (C_{-1},C_2)$,
and the dynamics is dissipative since the jacobian of $g_n$ in $[C_{-1},C_2]$ is everywhere less than $1-1/2n < 1$.
\ro{Since $f_2$ satisfies the twist condition which is open, the same holds for $g_n$ when $n$ is large. Thus $g_n$ presents Birkhoff attractors $\C_n$ for large $n\in\N$.

By the same arguments we did before, the closure of the unstable manifold $W^u(q,g_n)$ must be a circloid $\C'_n$ which is invariant for some power of $g_n$ .
As any power of $g_n$ is also a dissipative twist map, it has a unique invariant circloid, so it must be $\C'_n=\C_n$ (the same holds for $W^u(q',g_n)).$
Therefore the homoclinic class $H_n$ is contained in the Birkhoff attractor $\C_n$, so $\rho_{\C_n}(G_n)\supset [\delta,1-\delta]$ for some lift $G_n$ of $g_n$.}

On the other hand as $g_n$ can be considered in an arbitrary small $C^1$ \az{neighbourhood} of $\tau$, the entropy of $g_n$ can is arbitrary small
(Proposition \ref{rem-twist}), say smaller than $\eps/3$, and then\footnote{The iterate is just to ensure that the rotation set of a well chosen lift contains
$[0,1]$. Notice that the entropy of $g_n^3$ will be smaller than $\eps$.} choosing $g_n^3$ we obtain the proof of Theorem D.

\begin{obs}
It might be possible that the global attractor $\Lambda$ in this case is equal to the Birkhoff attractor $\C$. However, we did not find a simple argument \ve{to prove} this fact.
\end{obs}

\end{document}